\documentclass[reqno]{amsart}
\usepackage[mathscr]{eucal}
\usepackage{amsfonts}
\usepackage{amsmath, amsthm, amssymb}
\usepackage{latexsym}
\usepackage[dvips]{graphics}
\usepackage{color}
\usepackage{epsfig}

\theoremstyle{plain}
\newtheorem{thm}{Theorem}
\newtheorem{cor}[thm]{Corollary}
\newtheorem{lem}[thm]{Lemma}
\newtheorem{prop}[thm]{Proposition}

\theoremstyle{definition}

\theoremstyle{remark}

\numberwithin{thm}{section}

\numberwithin{equation}{section}

\def\be{\begin{equation}}
\def\ee{\end{equation}}
\def\bse{\begin{subequations}}
\def\ese{\end{subequations}}
\def\bge{\begin{eqnarray}}
\def\bgee{\begin{eqnarray*}}
\def\ege{\end{eqnarray}}
\def\egee{\end{eqnarray*}}

\def\R{{\Bbb R}}

\def\eps{\varepsilon}

\def \trait (#1) (#2) (#3){\vrule width #1pt height #2pt depth #3pt}
\def \fin{\hfill
        \trait (0.1) (5) (0)
        \trait (5) (0.1) (0)
        \kern-5pt
        \trait (5) (5) (-4.9)
        \trait (0.1) (5) (0)
\medskip}
\def \mfin{\trait (0.1) (5) (0)
        \trait (5) (0.1) (0)
        \kern-5pt
        \trait (5) (5) (-4.9)
        \trait (0.1) (5) (0)
}

\title{Grow-up rate and refined asymptotics for a two-dimensional Patlak-Keller-Segel 
model in a disk}

\author{Nikos I. Kavallaris}
\address{
\begin{flushleft}
        \hspace{0.3cm}   Department of Statistics and Actuarial-Financial Mathematics,\\ 
         \hspace{0.3cm}  University of the Aegean,\\
         \hspace{0.3cm}  Vourlioti Building,\\
         \hspace{0.3cm}  Gr-83200 Karlovassi, Samos, Greece\\
\end{flushleft}
}
\email{nkaval@aegean.gr}

\author{Philippe Souplet}
\address{
\begin{flushleft}
        \hspace{0.3cm}  Laboratoire Analyse G\'eometrie et Applications,\\
         \hspace{0.3cm}  UMR CNRS 7539, Institut Galil\'ee,\\
         \hspace{0.3cm}  Universit\'e Paris-Nord,\\
         \hspace{0.3cm}  99 av. J.-B. Cl\'ement,\\
         \hspace{0.3cm}  93430 Villetaneuse, France\\
\end{flushleft}
}
\email{souplet@math.univ-paris13.fr}

\begin{document}
\date{}
\maketitle
\thispagestyle{empty}
\pagestyle{myheadings}
\markboth{N. I. KAVALLARIS AND Ph. SOUPLET}{GROW-UP RATE FOR A PATLAK-KELLER-SEGEL MODEL}

%%%%%%%%%%%%%%%%%% Abstract form %%%%%%%%%%%%%%%

\begin{abstract}
We consider a special case of the Patlak-Keller-Segel system in a disc,
which arises in the modelling of chemotaxis phenomena.
For a critical value of the total mass, the solutions are known to be global in time but with density becoming unbounded,
leading to a phenomenon of mass-concentration in infinite time.
We establish the precise grow-up rate and obtain refined asymptotic estimates of the solutions. Unlike in most of the similar, recently studied, grow-up problems, the rate is neither 
polynomial nor exponential. In fact, the maximum of the density behaves like $e^{\sqrt{2t}}$ for large time. In particular, our study provides a rigorous proof of a behaviour suggested by Sire and Chavanis [Phys. Rev.~E, 2002] on the basis of formal arguments.

\bigskip


 \medskip
{\it $2000$ Mathematics Subject Classification{\rm :}}
 \; Primary  35K60, 35B40, 92C17;\, Secondary  35Q72.

%%%%%%%%%%%%%%%%%%%%%%%%%%%%%%%%%

\end{abstract}

%%%%%%%%%%%%%%%%%%Main Text%%%%%%%%%%%%%%%%%%%%%%

\section{Introduction}			\label{introd}

\subsection{The complete Patlak-Keller-Segel model}

%%%%    %%%%    %%%%    %%%%    %%%%    %%%%    %%%%    %%%%    %%%%    
Out of the many mathematical models that have been proposed to deal with particular aspects
of chemotaxis, that proposed by Patlak in 1953 (cf.~\cite{Pat}) and Keller and Segel in 1970 (cf.~\cite{KS1}) has received
particular attention. The so called (two-dimensional)  Patlak-Keller-Segel model consists of two equations,
describing the evolution of the population 
density $\rho(x,t)$ of bacteria,
and the concentration $c(x,t)$ of a chemical attracting substance, in a bounded domain $\Omega\subset \R^2$ and in a time interval $[0,T]$\,: 
\begin{eqnarray}
\frac{\partial{\rho}}{\partial t}&=&\nabla\cdot(D_1 \nabla\rho-\chi\rho \nabla c),\label{in1}\\
\theta\frac{\partial{c}}{\partial t}&=&\Delta c-ac+\rho. \label{in2}
\end{eqnarray}
More precisely the first equation describes the random (Brownian) diffusion
of the population of cells, which is 
biased in the direction of a drift velocity, proportional to the gradient of the concentration of the chemoattractant.
The diffusion coefficient is denoted by $D_1>0$
and the proportionality coefficient of the drift (mobility parameter) is denoted by $\chi>0$.
According to the second equation,
the chemoattractant, which is directly emitted by the cells, diffuses with a diffusion coefficient $D_2=1/\theta>0$ on the substrate, while is generated
proportionally to the density of cells and at
the same time is degraded with a rate equal to $a/\theta\geq 0.$ 
In order for system (\ref{in1})-(\ref{in2}) to be well posed it should be supplemented with some initial conditions
\begin{equation}\label{in3}
\rho(x,0)=\rho_0(x)\geq 0,\;\; c(x,0)=c_0(x)\geq 0,
\end{equation}
along with conditions on the boundary $\partial \Omega.$ 
A natural boundary condition, since it guarantees 
the conservation of total mass, is the no-flux type condition for $\rho$, namely  %REV1: (x,t) removed
\begin{equation}\label{in4}
\frac{\partial \rho}{\partial\nu}-\rho\frac{\partial c}{\partial\nu}=0\quad\mbox{on }\partial\Omega,
\end{equation}  
where $\nu$ stands for the outer unit normal vector at $\partial \Omega.$ 
As for $c$, a Dirichlet type boundary condition %REV1: (x,t) removed
is assumed  i.e.
\begin{equation}\label{in4a}
c=0\quad\mbox{on }\partial\Omega,
\end{equation} 
cf.~\cite{BN1,SC}.
Note that the parabolic system (\ref{in1})-(\ref{in2}) preserves the nonnegativity of the initial conditions,
i.e. $\rho,\,c\geq 0$ for $t>0$, which is also expected to be true for the physical problem. %REV1: (x,t) removed
For simplicity, $D_1, \chi$
are considered to be constant and under suitable scaling  
can be taken $D_1=\chi=1.$

In view of experimental facts,  
the coefficients $\theta$ and $a$ are assumed to be small
and a simplified form of the 
 Patlak-Keller-Segel system is obtained (in fact this corresponds to the case when
the diffusion and production of $c$ are much faster than the dynamics of $\rho$ and the degradation of $c$).
  Namely, by considering the limiting case 
$\theta,a\to 0+$,
the parabolic-parabolic system (\ref{in1})-(\ref{in4a}) is reduced to the elliptic-parabolic system
\begin{eqnarray}\label{in5}
\frac{\partial{\rho}}{\partial t}&=&\nabla\cdot( \nabla\rho-\rho \nabla c),
\quad x\in \Omega,\;t\in(0,T),\\
-\Delta c&=&\rho,
\quad x\in \Omega,\;t\in(0,T),\label{in6}\\
\frac{\partial \rho}{\partial\nu}-\rho\frac{\partial c}{\partial\nu}&=&0,
\quad x\in \partial\Omega,\;t\in(0,T),\label{in8}\\
c&=&0,\quad x\in \partial\Omega,\;t\in(0,T),\label{in8c}\\
\rho(x,0)&=&\rho_0(x)\geq 0,
\quad x\in \Omega.\label{in7}
\end{eqnarray}
Note that in order for 
(\ref{in5})-(\ref{in7}) to be well posed, 
only the initial data $\rho(x,0)=\rho_0(x)$ must be prescribed.
Moreover, owing to the boundary condition (\ref{in8}), the total (mass) population of cells is conserved, that~is
$$||\rho(\cdot,t)||_1=||\rho_0||_1=:\Lambda\quad\mbox{for } t>0.$$

For a more detailed analysis regarding the modelling 
as well as the behaviour of solutions of chemotaxis systems,
see the review papers \cite{H,Hor1,Hor2} as well as the monograph \cite{S}. Here, it should be noticed that system (\ref{in5})-(\ref{in7}) is also known as Smoluchowski-Poisson system and 
can describe 
the motion of the mean field of many self-gravitating particles \cite{Ch,Ba,B,BN1,SC,W1,W2} or that of polymer molecules \cite{DE}.
The behaviour of the solution to 
(\ref{in5})-(\ref{in7}),
strongly depends on
the parameter $\Lambda.$ In fact, if $\Lambda>8\pi$ and $\Omega=B(0,R),\, R>0,$ then solutions of 
(\ref{in5})-(\ref{in7}) blow up in a 
finite time $T^*(\rho_0),$ that is 
$$
\lim_{t\to T^*}||\rho(\cdot,t)||_{H^1}=\lim_{t\to T^*}||\rho(\cdot,t)||_{L^p}=\lim_{t\to T^*}\int_{\Omega} \rho\log \rho(x,t)dx=\infty
$$
for every $p>1,$ see \cite[Theorem 2(i)]{BN2}, \cite[Theorem 2]{B}. On the other hand for $\Lambda<8 \pi,$  
all solutions of system (\ref{in5})-(\ref{in7}) are global in time, 
cf.~\cite[Theorem 2 (iv)]{BN1}. In the critical case $\Lambda=8\pi$ an 
infinite-time blow-up (grow-up)
occurs, i.e. $||\rho(\cdot,t)||_{\infty}\to \infty$ as $t\to \infty,$ cf.~\cite[Theorem 3]{OSS}, \cite[Proposition 3.2]{BKLN1}.
Finite or infinite-time blow-up can 
be accompanied by the occurrence 
of a $\delta-$function formation in the blow-up set (this represents the trend
of populations to concentrate to form sporae) and is known in the literature as {\it chemotactic collapse}. This phenomenon was  conjectured by Childress and Percus \cite{CP}, Nanjundiah \cite{Nan}, 
and was first verified, via matched asymptotics arguments, for a radially symmetric simplified Patlak-Keller-Segel system in \cite{HV}. 
A result regarding the infinite-time Dirac mass formation for $\rho$ can 
be found in \cite{OSS}, where some characterization of grow-up (mass-concentration) points together with
more grow-up results for different types 
of boundary conditions are also obtained, \cite[Theorem 2 \& Theorem 3]{OSS}. 
For blow-up results  
concerning a variation of system (\ref{in5})-(\ref{in7})
but with Neumann boundary conditions for both $\rho$ and $c,$ see \cite{JL,Na,B1,Na1,SeSu}.

\subsection{The simplified Patlak-Keller-Segel model in the radial case}
In the case when $\Omega$ is the ball $B(0,R),\; R>0,$ and the initial data
$\rho_0(x)=\rho_0(r)$ is radially symmetric, the solution of 
system (\ref{in5})-(\ref{in7}) is radially symmetric, i.e.: 
$\rho(x,t)=\rho(r,t),$ with $r=|x|.$ 
In this case the elliptic-parabolic system (\ref{in5})-(\ref{in7}) 
can be greatly simplified. 
Namely, by introducing the cumulative mass
distribution 
$$
Q(r,t):=\int_{B(0,r)}\rho(x,t)dx=2\pi\int_0^r s\rho(s,t)ds,
$$
which is equal to the mass contained in the sphere $B(0,r)$, the system reduces to a single equation 
\begin{eqnarray}
Q_t&=&Q_{rr}-\frac{1}{r}Q_{r}+\frac{1}{2\pi r}QQ_r,\quad 0<r<R,\;t>0,\label{in11}\\
Q(0,t)&=&0,\quad Q(R,t)=\Lambda.\label{in12}
\end{eqnarray}
By the definition of $Q$, the function 
\be Q(r,0)=Q_0(r)\label{in12a}\ee
 is positive nondecreasing and
satisfies the compatibility conditions $Q_0(0)=0$ and $Q_0(R)=\Lambda.$

As mentioned in \cite{BHN,BN1} the formulation (\ref{in11})-(\ref{in12a}) allows the consideration
of some initial data for the density $\rho$ which could be either unbounded or singular (such as measures), a case that seems rather realistic. %REV1: (x,t) removed
This means that the initial data $Q_0$ %REV1: r removed
for problem (\ref{in11})-(\ref{in12a}) could have unbounded derivatives $Q_{0,r}$ or even be discontinuous. Moreover, using formulation (\ref{in11})-(\ref{in12a}) we have the comparison principle at hand, which is not available for system (\ref{in5})-(\ref{in7}).
Due to the scaling properties of 
equation (\ref{in11}), we can assume without loss of generality that problem (\ref{in11})-(\ref{in12a}) is posed in 
the unit ball $B(0,1).$ (Indeed, it is easily seen that if $Q(r,t)$ is a solution of (\ref{in11})-(\ref{in12a}) then
 $Q(Rr,R^2t)$ is also a solution.)

Using the new variable $x=r^2$ 
(no confusion with the original variable $x$ in (\ref{in5})-(\ref{in7}) should occur)
and defining $N(x,t)=Q(r,t)$, we are led to
the problem
\bge
N_t&=&4xN_{xx}+\frac{1}{\pi}N N_x,\quad 0<x<1,\;t>0,\label{in13}\\
N(0,t)&=&0,\quad N(1,t)=\Lambda,\label{in14}\\
N(x,0)&=&N_0(x),\quad 0<x<1.\label{in15}
\ege
Note that (\ref{in13}) differs from the Burgers equation only by the variable coefficient $x$
in the diffusion term. 
The above problem, although it
 degenerates at $x=0,$ may be handled more easily 
  than (\ref{in11})-(\ref{in12a}) since it contains less terms, and at the same time does not have any singular coefficients in the first order terms.

As is expected, the behaviour of the solution of problem (\ref{in13})-(\ref{in15}) 
(which is well defined for suitable initial data) depends on $\Lambda.$ 
For $\Lambda> 8\pi$ the solution $N$ ceases to exist in a finite time $T^*$,  %REV1: (x,t) removed
 more precisely the boundary condition $N(0,t)=0$ is no longer fulfilled at $t=T^*.$
Moreover a ``gradient blow-up" occurs at $t=T^*$, in the sense that $N_x(0,t)\to\infty$ as $t\to T^*$ 
and the density $\rho$ %REV1: (x,t) removed
becomes also unbounded at time $T^*,$ cf.~\cite[Theorem 2(i)]{BN3}. On the other hand, for $0<\Lambda<8 \pi$ and any (admissible) initial data
there is a unique global-in-time solution $N\in C([0,\infty);L^2(0,1))\cap C^{2,1}((0,1)\times(0,\infty))$.
Furthermore, $N$ converges to the unique steady state solution~:
\be\label{conv}
N(\cdot,t)\to N_d=8 \pi\frac{x}{x+d},\quad\mbox{as } t\to \infty
\ee 
in $L^p(\Omega),\;p\geq 1$, and even in $L^{\infty}(\Omega)$ provided that
$\sup_{t\geq 0}|N_x|_{\infty}<\infty,$ where $d=\frac{8 \pi}{\Lambda}-1>0,$ cf.~\cite{BKLN1}. 
In this case the rate of the $L^1$-convergence 
of $N(\cdot,t)$ to $N_d$ is shown to be exponential.

In the borderline case $\Lambda=8 \pi,$ the situation is still different and the problem exhibits a typical critical behaviour. 
Namely, it is proved in \cite{BKLN1} 
that there exists a global-in-time solution $N$,  %REV1: (x,t) removed
which converges to the ``singular" steady state $\widetilde{N}(x)\equiv 8\pi$ 
(note that $\widetilde{N}$ does not satisfy the boundary condition at $x=0$).
Actually, as was proven in \cite[Theorem 3]{OSS}, an infinite-time Dirac mass formation at the origin $r=0$ of the ball occurs in this case. However, neither an estimation of the grow-up rate nor the asymptotic profile of the grow-up are provided in \cite{OSS}.
On the other hand the authors in \cite[Proposition 3.2]{BKLN1} obtain the decay estimate
\be\label{in16}
||N(\cdot,t)-8\pi||_{L^1}\leq \frac{8 \pi}{t}\quad\mbox{for}\quad t\geq 1.
\ee
Estimate (\ref{in16}) seems to be far from optimal 
since formal asymptotics performed in \cite{SC} suggest 
a temporal decay estimate of order
\be\label{in17}
||N(\cdot,t)-8\pi||_{L^1}\approx O\bigl(e^{-\sqrt{2t}}\bigr),\quad\mbox{as } t\to \infty.
\ee

\medskip
{\bf Remarks 1.1.} (a)
The (nonradial) Patlak-Keller-Segel system has also been studied in the whole plane $\R^2$.
Again, the behaviour of solutions depends on the initial mass of the system and a dichotomy is found \cite{BDP,DP}.
More precisely, assuming $0\leq (1+|x|^2)\rho_0$ and $\rho_0\log \rho_0\in L^1$, there exists a critical value of the mass $N_c := 8\pi$ such that if 
$0 < ||\rho_0||_1< N_c$ (subcritical
case) only global-in-time solutions exist, while if $||\rho_0||_1 > N_c$ (supercritical case) the solutions blow
up in finite time \cite{BDP,P}. Moreover, in the subcritical case, solutions converge to a self-similar profile as $t\to\infty$
\cite{BKLN2,BDP}. 
Finally, for the critical case $N = N_c$, which was studied in \cite{BCM}, the solution is global-in-time and grows up as a Dirac mass at the centre of mass as $t\to\infty.$
\smallskip

(b) The only previous mathematical study of grow-up rates for a system of Patlak-Keller-Segel type concerns high
dimensions, namely $n\geq 11$, and was performed recently in \cite{Se}. There, some radial global unbounded solutions were constructed in a ball and an infinite sequence of polynomial grow-up rates was obtained (for a suitable sequence of initial data). On the contrary, our results in the present paper for $n=2$ exhibit a grow-up rate independent of the initial data,
\smallskip

(c) Concerning the parabolic-parabolic Patlak-Keller-Segel system, in a bounded domain with Neumann conditions, interesting results can be found in \cite{HV2} and \cite{FHL}, respectively on the asymptotics of finite-time blow-up, and on the convergence of bounded solutions.

\smallskip

(d) Problem (\ref{in13})-(\ref{in15}) is in fact independent of the boundary condition 
(\ref{in8c}) for $c$.
Of course the boundary condition for $c$ has to be taken into account if one wants to 
determine $c$, and not only $\rho$, from~$N$.

\smallskip

(e) In the case where the diffusion of cells is very slow as compared to the diffusion of the chemoattractant, i.e. when $D_1\ll D_2$ in (\ref{in1})-(\ref{in2}), the complete system (\ref{in1})-(\ref{in4a}) is reduced to a single but non-local equation, cf.~\cite{W3}.
The global existence (subcritical case) and the finite-time blow-up (supercritical case) of solutions of the derived non-local equation in the two-dimensional case is studied in \cite{KS}.

\subsection{Heuristic description of the results and methods}
Our aim is to prove rigorously the decay rate (\ref{in17}), as well as to provide a refined asymptotic
profile for $N(x,t)$ as $t\to \infty$.
To normalize the constants arising in 
the calculations we shall 
work with the equivalent problem 
\bge
u_{t}&=&xu_{xx}+2u u_x,\qquad 0<x<1,\;t>0,\label{in13a}\\
u(0,t)&=&0,\qquad u(1,t)=\xi:=\frac{\Lambda}{8 \pi},\qquad t>0,\label{in14a}\\
u(x,0)&=&u_0(x),\qquad 0<x<1,\label{in15a}
\ege
which is obtained from (\ref{in13})-(\ref{in15}) by setting 
$N(x,t)=8\pi u(x,4t)$.

It is clear that the stationary part $xu_{xx}+2uu_x=0$  
of the parabolic equation (\ref{in13a}) is invariant under the rescaling $u(x)\mapsto u(kx)$ ($k>0$).
Moreover, the steady state solution for $\xi<1$ is given by
$$ %%REV2: instead of \be\ee
U_{a}(x)=1-\frac{1}{ax+1}=\frac{ax}{ax+1}
$$
for $a=U'_{a}(0)=\frac{\xi}{1-\xi}>0$. Rewriting 
(\ref{conv}) in terms of $u$ we have
$$
u(x,t)\to U_{a}(x),\quad\mbox{as } t\to \infty.
$$ 
Also, observe that $U_{a}(x)$ converges in a monotone increasing way
to the ``singular" steady state $U\equiv 1$ as $a\to \infty\;(\xi\to 1).$ 

Motivated by the above considerations, we shall look for 
 upper and lower solutions of problem (\ref{in13a})-(\ref{in15a}),
which are perturbations of a moving family of steady states 
\be\label{quasistat}
U_{a(t)}(x),
\ee 
the perturbation being defined in terms of the self-similar variable $y=a(t)x$. 
Here, $a$ is a function of time, diverging to infinity,
which is a priori unknown
and will be eventually identified by a suitable ``matching'' procedure
(see Remark~1.3).
Note that such a form was used in \cite{SC} to construct ``approximate solutions'' leading to the formal asymptotics mentioned above.
However, the expansions in \cite{SC} contained only a correction term at the first order, while those that we here construct involve first and second order terms
(see Lemmas~\ref{LemMain1} and \ref{LemMain2}). 
This seems necessary to obtain (rigorous) upper and lower solutions
living ``close"  to the actual solution, and eventually providing us with the desired decay rate,
as well as with a good description of the asymptotic profile.

As a consequence of this construction
we shall 
derive an infinite-time boundary 
gradient grow-up result, 
with the grow-up rate
\be\label{in17b}
u_x(0,t)=A(t)\Bigl(1+O\bigl(t^{-1/2}\log t\bigr)\Bigr),
\quad\mbox{as } t\to\infty,\qquad\hbox{where }ÊA(t)=\exp\left[\frac{5}{2}+\sqrt{2t}\right]. 
\ee
 In terms of the solution of the original system (\ref{in5})-(\ref{in7}),
(\ref{in17b}) gives an estimate of the central density of bacteria,
since $\rho(0,t)=8u_x(0,4t)$.
Note that (\ref{in17b}) was also predicted 
 by formal arguments in \cite{SC}.
At the same time, we show the $C^1$ regularity of $u$ %REV1: (x,t) removed
 up to the boundary for all finite time
 intervals, 
hence ruling out the possibility of $\sup_{x\in[0,1]}u_x(x,t)$ blowing-up in finite time. 
This problem was left open in \cite{BKLN1}.
Moreover, we obtain a precise asymptotic expansion of the solution (see formula (\ref{expansionmainthm}) in Theorem~\ref{Theorem 1}).
It expresses the solution as the sum of a quasi-stationary profile (cf.~(\ref{quasistat})) and of a correction term
which becomes significant only for $x$ bounded away from $0$.
As a consequence, we obtain the decay 
$$
\|u(\cdot,t)-1\|_{L^1(0,1)}=\Bigl(1+O\bigl(t^{-1/2}\log t\bigr)\Bigr)\sqrt{2t}\,\exp\left[-\frac{5}{2}-\sqrt{2t}\right],
\quad\mbox{as } t\to\infty,$$
again in accordance with the predictions 
in \cite{SC}.

\medskip

{\bf Remark 1.2.} The study of unbounded global classical solutions of superlinear parabolic problems and their asymptotic behaviour
has recently attracted substantial mathematical interest.
Particular effort has been devoted to the reaction-diffusion equation 
\be\label{eqnup}
u_t-\Delta u=u^p,
\ee
where such solutions are known to exist for suitable $p>1$. 
Let us mention the works \cite{LT, DGLV, GK} for the Dirichlet problem in a ball, and \cite{PY, FWY, FKWY, Mi} for the Cauchy problem. 
See also \cite[Sections~22 and 29]{QS} for related questions. The case of the Frank-Kamenetskii equation (with nonlinearity $e^u$ instead of $u^p$) is also studied in \cite{DGLV}.
As for the diffusive Hamilton-Jacobi equation 
\be\label{eqnVHJ}
u_t-u_{xx}=|u_x|^p,
\ee
with $p>2$, results of this kind can be found in \cite{SV}.
A common feature in all these examples is the stabilization of the solution to a singular steady-state, the
growing-up quantity being $\|u(t)\|_\infty$ (resp., $\|u_x(t)\|_\infty$) for equation (\ref{eqnup}) (resp., (\ref{eqnVHJ})).
It has to be noted that the grow-up rates, as $t\to\infty$, usually behave either like $e^{\mu t}$ or $t^k$.
The only known exception (see \cite{GK}) seems to be the case of equation (\ref{eqnup}) 
with zero boundary conditions, in a ball of $\R^4$ with critical Sobolev exponent ($p=3$).
In this situation, unlike in spatial dimensions $n\neq 4$, there holds
$\log\|u(t)\|_\infty\sim 2\sqrt t$, thus leading to a similar rate as in our problem.

\medskip

{\bf Remark 1.3.} Let us point out that our determination of the grow-up rate 
for problem (\ref{in13a})-(\ref{in15a}) is achieved
through the matching of (sub- or super-) solutions with the imposed 
boundary condition at the right end-point $x=1$, an idea
 also present in \cite{GK}.
This is different from what is done in \cite{DGLV,FWY, FKWY,SV}, where
the grow-up rate is determined by the matching between inner and outer (sub-/super-) solutions.
In those works, the inner solution corresponds to a self-similar, quasi-stationary evolution along a continuum of regular steady states
(similar to (\ref{quasistat})), but
the outer solution is obtained by a linearization around the singular steady state.
Here, on the contrary, the behaviour of $u$  %REV1: (x,t) removed
 in the inner and outer regions is unified through a single self-similar variable $y=a(t)x$
(cf.~formulae~(\ref{AnsatzUnderu}) and (\ref{AnsatzOveruA}) below).
Moreover we point out that, in our case, a linearization around the ``singular" steady
state $U\equiv 1$ would not give the desired grow-up rate (\ref{in17b}), but only a non-optimal exponential upper bound
based on an associated eigenvalue problem.

\medskip

The paper is organized as follows.  
In Section~2 we state the main results. Section~3 contains a number of preliminary results:
In Subsection~3.1, we recall the basic facts concerning local existence
and comparison. 
In Subsections~3.2 and 3.3, we show that a control on the slope at $x=0$ is enough to 
prevent gradient blow-up (Lemma~\ref{LemC1}), and we obtain 
preliminary estimates of solutions for small time, which will be useful in Section~4 to initialize the comparison with the main sub-/supersolutions (Lemmas~\ref{LemInit1} and \ref{LemInit2}). Subsection~3.4 is devoted to
the study of a second order ordinary differential operator which plays a key role in the subsequent construction of sub-/supersolutions.
The proofs of the main results are given in Section~4:
Subsections~4.1 and 4.2 are devoted to the construction of the main sub- and supersolutions respectively;
The proofs of Theorem~\ref{Theorem 1} and Corollary~\ref{Corollary 2} are finally completed in Subsection~4.3.
%%%%    %%%%    %%%%   
 %%%%    %%%%    %%%%    %%%%    %%%%    %%%%    
 
%%Sec main results
\section{Main results}
Consider the problem
\bge
u_t-xu_{xx}&=& 2uu_x,\quad 0<x<1,\ t>0, \label{pbmu1}\\
u(0,t)&=&0,\quad t>0, \label{pbmu2}\\
u(1,t)&=& 1,\quad t>0, \label{pbmu3}\\
u(x,0)&=&u_0(x),\quad 0\leq x\leq 1. \label{pbmu4}
\ege
Concerning the initial data, we assume that 
\be\label{hypIDa}
u_0\in C([0,1]),\quad u_0(0)=0,\quad u_0(1)=1,
\quad\hbox{$u_0$ is nondecreasing,}
\ee
and that
\be\label{hypIDb}
u_0(x)\leq Kx,\ \ 0<x<1, \quad\hbox{ for some $K\geq 1$}.
\ee
Problem (\ref{pbmu1})-(\ref{pbmu4}) admits a unique global solution, with $u\in C([0,1]\times [0,\infty))$, $u\in C^{2,1}((0,1]\times (0,\infty))$
(see \cite{BKLN1} and Section~3.1 below).
Moreover it was shown in \cite{BKLN1} that $0\leq u\leq 1$, $u$ is nondecreasing in $x$, and 
\be\label{CvToOne}
\lim_{t\to\infty} u(x,t)=1,\quad\hbox{uniformly for $x$ in compact subsets of $(0,1]$.}
\ee

Our main results are the following:

\begin{thm}
\label{Theorem 1}
Let $u_0$ satisfy (\ref{hypIDa}) and (\ref{hypIDb})
and denote by $u$ the global solution of problem (\ref{pbmu1})-(\ref{pbmu4}).
Then there holds
\be\label{expansionmainthm}
1-u(x,t)={1-x+O\bigl(t^{-1/2}\log t\bigr)\over 1+A(t)x},
\quad\mbox{uniformly in $[0,1]$, as $t\to\infty$,}
\ee
with
\be\label{amainthm}
A(t)=\exp\left[\frac{5}{2}+\sqrt{2t}\right].
\ee
Moreover, we have the regularity property
\be\label{reguxThm}
u_x\in C\bigl([0,1]\times (0,\infty)\bigr)
\ee
and the estimate
\be\label{CentralDensity}
u_x(0,t)=A(t)\Bigl(1+O\bigl(t^{-1/2}\log t\bigr)\Bigr),
\quad\mbox{as } t\to \infty.
\ee
\end{thm}

\begin{cor}
\label{Corollary 2}
Let $u_0$ satisfy (\ref{hypIDa}) and (\ref{hypIDb}),
denote by $u$ the global solution of problem (\ref{pbmu1})-(\ref{pbmu4}),
and let $A(t)$ be given by (\ref{amainthm}).
\smallskip

\noindent (i) The solution $u$ satisfies
the inner layer expansion (quasi-stationary behaviour):
$$1-u(x,t)={1+o(1)\over 1+A(t)x},
\quad\hbox{uniformly in any region $x\leq o(1)$, as $t\to\infty$.}$$
\smallskip

\smallskip
\noindent (ii) We have the $L^1$-decay rate
$$ %%REV2: instead of \be\label{MassAsympt} (not used)
\|u(\cdot,t)-1\|_{L^1(0,1)}=\Bigl(1+O\bigl(t^{-1/2}\log t\bigr)\Bigr)\sqrt{2t}\,\exp\left[-\frac{5}{2}-\sqrt{2t}\right],
\quad\mbox{as } t\to\infty.
$$
\end{cor}

%%END Sec main results

\section{Preliminaries}	\label{pre}
%%%%    %%%%    %%%%    %%%%    %%%%    %%%%    %%%%    %%%%    %%%%    

\subsection{Local existence and comparison principle.}

By \cite[Theorem~2.1]{BKLN1}, we know that for any $u_0$ satisfying (\ref{hypIDa}), 
problem (\ref{pbmu1})-(\ref{pbmu4}) admits a global solution $u$ in the following sense:
\be\label{reguA}
u\in C\bigl([0,\infty);L^1(0,1)\bigr),
\ee
\be\label{reguB}
u\in C^{2,1}\bigl((0,1]\times (0,\infty)\bigr),
\ee
\be\label{reguC}
u_x\geq 0,\quad 0<x\leq 1,\ t>0,
\ee
and
\bge
u_t-xu_{xx}&=& 2uu_x,\quad 0<x<1,\ t>0, \label{reguD1}\\
\lim_{x\to 0^+}u(x,t)&=&0,\quad\hbox{for a.e. $t>0,$} \label{reguD2}\\
u(1,t)&=& 1,\quad t>0, \label{reguD3}\\
u(\cdot,0)&=&u_0\quad\hbox{in $L^1(0,1)$}. \label{reguD4}
\ege
Note that the solution is obtained in \cite{BKLN1} as limit of solutions of regularized problems,
where $xu_{xx}$ is replaced by $(x+\eps)u_{xx},\; \eps>0.$
Also, for each $T>0$, $u$ is the unique local solution
of (\ref{reguA})-(\ref{reguD4}) on $(0,T)$.
If, moreover, $u_0$ satisfies (\ref{hypIDb}), then
$$ %%REV2: instead of \be\label{reguCont} (not used)
u\in C\bigl([0,1]\times [0,\infty)\bigr),
$$
see also \cite[Theorem 1 (i)]{BN3}.
The continuity for $t>0$ and $x=0$ follows from \cite[Propositions~2.4 and 2.5]{BKLN1}. 
For $t=0$ and $x\in [0,1]$, the continuity can be established
by comparison with simple barrier functions.

\medskip
The following proposition provides a comparison principle suitable to our needs.

\begin{prop}
\label{PropComp}
Let $\tau>0$ and the functions $u,v$ satisfy the following regularity conditions:
\be\label{reguAcomp}
u, v\in C\bigl([0,\tau);L^1(0,1)\bigr),
\ee
\be\label{reguBcomp}
u, v\in C^{2,1}\bigl((0,1]\times (0,\tau)\bigr),
\ee
\be\label{reguCcomp}
u_x, v_x\in L^1_{loc}\bigl([0,1]\times (0,\tau)\bigr),
\ee
\be\label{reguC2comp}
u, v\in L^\infty_{loc}\bigl([0,1]\times (0,\tau)\bigr).
\ee
Assume that
\bge
u_t-xu_{xx}-2uu_x&\leq &v_t-xv_{xx}-2vv_x,\quad 0<x<1,\ 0<t<\tau, \label{reguD1comp}\\
\lim_{x\to 0^+}u(x,t)&\leq &\lim_{x\to 0^+}v(x,t),\quad\hbox{for a.e. $t\in(0,\tau),$} \label{reguD2comp}\\
u(1,t)&\leq &v(1,t),\quad 0<t<\tau, \label{reguD3comp}\\
u(\cdot,0)&\leq &v(\cdot,0)\quad\hbox{a.e. in $(0,1)$}. \label{reguD4comp}
\ege
Then $u\leq v$ in $(0,1)\times (0,\tau)$.
\end{prop}
Observe that (\ref{reguCcomp}) implies that the limits in (\ref{reguD2comp}) exist for a.e.~$t\in (0,\tau)$.
Note also that for any $u_0$ satisfying (\ref{hypIDa}), the solution $u$ of problem (\ref{reguA})-(\ref{reguD4}) satisfies conditions
 (\ref{reguCcomp}) and (\ref{reguC2comp}), 
as a consequence of (\ref{reguC}), (\ref{reguD2}) and (\ref{reguD3}).

\medskip
{\it Proof of Proposition~\ref{PropComp}.} It is a modification of the stability proof in \cite[Theorem~3.1]{BKLN1}. Let $z=u-v$. By (\ref{reguD1comp}), we have
\be\label{phs1}
z_t\leq xz_{xx}+2uu_x-2vv_x=\partial_x\bigl(xz_x+z(u+v-1)\bigr),\quad 0<x<1,\ 0<t<\tau.
\ee
For $\delta\in(0,1)$ we define the following $C^1$ (and piecewise $C^2$)
convex approximations of the function $s\to s_+=\max(s,0)$:
$$\phi_\delta(s)=
\begin{cases}
0, \quad\hbox{ if $-\infty<s\leq \delta$, }\\
\noalign{\vskip 1mm}
(2\delta)^{-1}(s-\delta)^2, \quad\hbox{ if $\delta\leq s\leq 2\delta$, }\\
\noalign{\vskip 1mm}
s-3\delta/2, \quad\hbox{ if $2\delta<s<\infty$. }\\
\end{cases}$$

Fix $0<t_1<t_2<\tau$ and $\delta\in(0,1)$. Then, for any $\eps\in(0,1)$, 
multiplying (\ref{phs1}) by $\phi'_\delta(z),$ integrating by parts and using (\ref{reguC2comp}), (\ref{reguD3comp}), $0\leq \phi'_\delta\leq 1$
and $\phi''_\delta\geq 0$, we obtain
\begin{eqnarray*}
&&\int_\eps^1\phi_\delta(z(x,t_2))\,dx-\int_\eps^1\phi_\delta(z(x,t_1))\,dx\\
&=&\int_{t_1}^{t_2}\int_\eps^1 \phi'_\delta(z)z_t\,dxdt
\leq\int_{t_1}^{t_2}\int_\eps^1 \phi'_\delta(z)\partial_x\bigl(xz_x+z(u+v-1)\bigr)\,dxdt\\
&=&\int_{t_1}^{t_2}\Bigl[ \phi'_\delta(z)\bigl(xz_x+z(u+v-1)\bigr)\Bigr]_\eps^1\,dt
-\int_{t_1}^{t_2}\int_\eps^1 \phi''_\delta(z)\bigl(xz_x+z(u+v-1)\bigr)z_x\,dxdt\\
&\leq&\eps\int_{t_1}^{t_2}|z_x(\eps,t)|\,dt+C\int_{t_1}^{t_2}\phi'_\delta(z(\eps,t))\,dt
-\int_{t_1}^{t_2}\int_\eps^1 \phi''_\delta(z)(u+v-1)zz_x\,dxdt\equiv I_\eps+J_\eps+K_\eps,
\end{eqnarray*}
where $C$ depends on $t_1,t_2$ but is independent of $\eps$ (and $\delta$).
Owing to (\ref{reguCcomp}), there exists a sequence $\eps_n\to 0+$ such that $\lim_{n\to\infty}I_{\eps_n}=0$.
Next, since $\lim_{n\to\infty}\phi'_\delta(z(\eps_n,t))=0$ for a.e. $t\in (t_1,t_2)$ due to (\ref{reguD2comp}) and
the definition of $\phi_\delta$,
we deduce that $\lim_{n\to\infty}J_{\eps_n}=0$ by dominated convergence.
Consequently,
$$\int_0^1\phi_\delta(z(x,t_2))\,dx-\int_0^1\phi_\delta(z(x,t_1))\,dx
\leq -\int_{t_1}^{t_2}\int_0^1 \phi''_\delta(z)(u+v-1)zz_x\,dxdt
\leq C\int_{t_1}^{t_2}\int_0^1 |z\phi''_\delta(z)||z_x|\,dxdt.$$

Now, observe that $\lim_{\delta\to 0}\phi_\delta(s)=s_+$ and $\lim_{\delta\to 0}s\phi''_\delta(s)=0$ for each $s\in\R$.
Using $0\leq\phi_\delta(s)\leq s_+$, (\ref{reguCcomp}) and (\ref{reguC2comp}),
we may pass to the limit $\delta\to 0$ by dominated convergence in the preceding inequality and we obtain
$$\int_0^1 z_+(x,t_2)\,dx-\int_0^1z_+(x,t_1)\,dx\leq 0.$$
Letting $t_1\to 0+$ and using (\ref{reguAcomp}) and (\ref{reguD4comp}), 
we conclude that $\int_0^1 z_+(x,t_2)\,dx=0$ for all $t_2\in (0,\tau)$,
hence $u\leq v$ in $(0,1)\times (0,\tau)$.
\fin

\subsection{Sufficient condition for $C^1$ regularity.}
As noted in, e.g., \cite[Section~2.2]{BCKSV} and in \cite[Section 2]{GP}, by means of the transformation
$$w(r,t):=\frac{8}{r^2}u(r^2,4t)=\frac{1}{\pi r^2}\int_{B_r}\rho(y,4t)\,dy$$
(and $w_0(r):=8r^{-2}u_0(r^2)$),
problem (\ref{pbmu1})-(\ref{pbmu4}) becomes equivalent to
\bge
w_t-\tilde\Delta w&=& w^2
+\textstyle{r\over 2}{\hskip 1pt}w{\hskip 1pt}w_r,\quad 0<r<1,\ t>0, \label{pbmw1}\\
w_r(0,t)&=&0,\quad t>0, \label{pbmw2}\\
w(1,t)&=& 8,\quad t>0, \label{pbmw3}\\
w(r,0)&=&w_0(r),\quad 0<r<1, \label{pbmw4}
\ege
where $\tilde\Delta w:=w_{rr}+{3\over r}w_r$, which in turn corresponds to the radial Laplacian in $4$ space dimensions. 
It should be noticed that $w$ has the %REV1: (r,t), (x,t)  removed (4 times)
same scale invariance as $\rho$, but is smoother than $\rho$. Problem (\ref{pbmw1})-(\ref{pbmw4}) turns out to be convenient regarding the study  of  $C^1$ regularity of $u$ up to the boundary, which was left open in \cite{BKLN1}. 
Namely, using this transformation, one can show
the following additional properties for $u$.

\begin{lem}
\label{LemC1}
Let $u_0$ satisfy (\ref{hypIDa}) and (\ref{hypIDb})
and denote by $u$ the global solution of problem (\ref{pbmu1})-(\ref{pbmu4}).
For any given $T>0$, if
\be\label{hypBddSlope}
\sup_{0<x<1,\ 0<t<T}{u(x,t)\over x}<\infty,
\ee
then
\be\label{reguxA}
u_x\in C\bigl([0,1]\times (0,T]\bigr)
\ee
and
\be\label{reguxB}
u_x(0,t)>0,\quad 0<t\leq T.
\ee
\end{lem}

Note that the assumption (\ref{hypBddSlope}) (for all $T>0$)
will be shown in Section~4 (see Lemma~4.3(ii)), as a consequence of our main
upper solution construction, hence leading to the global $C^1$ regularity property (\ref{reguxThm}) in Theorem~\ref{Theorem 1}.

\medskip

{\it Proof of Lemma~\ref{LemC1}.} Since the semilinear parabolic equation in (\ref{pbmw1}) has only linear growth with respect to the gradient, standard arguments based on the variation-of-constants formula (see, e.g.,~\cite[Example~51.30]{QS},
or \cite[p.~889]{CW}) show that problem 
(\ref{pbmw1})-(\ref{pbmw4}) is locally well-posed in the space of (radial, nonnegative) 
$L^\infty$ functions.
More precisely, for any $0\leq w_0\in L^\infty(0,1)$, there exists a unique, maximal (radial, nonnegative) classical solution $w$ of (\ref{pbmw1})-(\ref{pbmw4}), with
$w\in C^{2,1}\bigl([0,1]\times (0,T_m)\bigr)$ and
$w\in C\bigl([0,T_m);L^q(0,1)\bigr)$ for all finite $q\leq 1$.
Moreover,
\be\label{BUaltw}
T_m<\infty\Longrightarrow \lim_{t\to T_m}\|w(t)\|_\infty=\infty.
\ee

Now, if $u_0$ satisfies (\ref{hypIDa})-(\ref{hypIDb}), then
$w_0(r):=8r^{-2}u_0(r^2)$ verifies $0\leq w_0\in L^\infty(0,1)$. 
Denote by $w$ the corresponding maximal solution of 
(\ref{pbmw1})-(\ref{pbmw4}) and let
$$\tilde u(x,t)={x\over 8}w(\sqrt{x},t/4),\qquad 0\leq x\leq 1,\ 0\leq t<4T_m.$$
We see that $\tilde u$ satisfies the regularity conditions in (\ref{reguAcomp}), (\ref{reguBcomp}) and
(\ref{reguC2comp}) with $\tau=4T_m$. Also, since
$$\tilde u_x(x,t)={1\over 8}w(\sqrt{x},t/4)+{\sqrt{x}\over 16}\,w_r(\sqrt{x},t/4),
\qquad 0<x\leq 1,\ 0<t<4T_m,$$
we have 
\be\label{regulux}
\tilde u\in C^{1,0}([0,1]\times (0,4T_m)),
\ee
hence in particular (\ref{reguCcomp}) with $\tau=4T_m$.
Then one easily checks that $\tilde u$ solves (\ref{pbmu1})-(\ref{pbmu3}) on $[0,4T_m)$,
along with (\ref{reguD4}).
We may thus apply Proposition~\ref{PropComp} to deduce that
$\tilde u=u$ on $(0,T_0)$ with $T_0=\min(T,4T_m)$. 

We claim that $T_m>T/4$.
Indeed, if $T_m\leq T/4$, then (\ref{BUaltw}) implies
$$\lim_{t\to 4T_m} \ \ \sup_{0<x<1}{u(x,t)\over x}=\infty,$$ 
contradicting (\ref{hypBddSlope}).
Consequently, property (\ref{reguxA}) follows from (\ref{regulux}).
Moreover, $w>0$ in $[0,1]\times (0,T_m)$ by the strong maximum principle, 
which readily implies (\ref{reguxB}). \fin

\subsection{Small time estimates.}
Throughout the paper, we denote by ${\mathcal P}$ the parabolic operator defined by
\be\label{DefOperP}
{\mathcal P}v:=v_t-xv_{xx}-2vv_x.
\ee

The following two lemmas will be useful to initialize the comparison between $u$ 
and the main lower/upper solutions constructed in Section~4.

\begin{lem}
\label{LemInit1} 
Let $u_0$ satisfy (\ref{hypIDa}) and (\ref{hypIDb})
and denote by $u$ the global solution of problem (\ref{pbmu1})-(\ref{pbmu4}).
Then there exist $\tau,\eta>0$ such that
\be\label{Est1Init1}
u(x,t)\leq 2Kx,\qquad 0\leq x\leq 1,\ 0\leq t\leq\tau,
\ee
and
$$u(x,\tau)\leq 1-\eta(1-x),\qquad 0\leq x\leq 1.$$
\end{lem}

{\it Proof.} Define 
$$\overline v(x,t)={Kx\over 1-2Kt},\qquad 0\leq x\leq 1,\ 0\leq t<1/2K.$$ 
Since
$${\mathcal P}\overline v={2K^2x\over (1-2Kt)^2}-{2K^2x\over (1-2Kt)^2}=0$$
and $\overline v(1,t)\geq K\geq 1$, the comparison principle guarantees that
$$u(x,t)\leq \overline v(x,t)\leq 2Kx,\qquad 0\leq x\leq 1,\ 0<t\leq 1/4K.$$ 
Fix $\tau=1/4K$. By Hopf's lemma, we have $u_x(1,\tau)>0$.
Since $u(\cdot,\tau)$ is nondecreasing in $x$, this implies 
$u(x,\tau)\leq 1-\eta(1-x)$, $0\leq x\leq 1$, for $\eta>0$ sufficiently small and the conclusion follows.  \fin

\bigskip

\begin{lem}
\label{LemInit2} 
Let $u_0$ satisfy (\ref{hypIDa}) and (\ref{hypIDb})
and denote by $u$ the global solution of problem (\ref{pbmu1})-(\ref{pbmu4}).
For any given $\delta\in (0,1)$,
there exists $T_\delta>0$ such that
$$u(x,T_\delta)\geq \min(1-\delta,x/\delta),\qquad 0\leq x\leq 1.$$
\end{lem}

\medskip

{\it Proof.}  By Lemma~\ref{LemC1}, we may fix a small $\tau>0$
such that $u_x(\cdot,\tau)\in C([0,1])$ and $u_x(0,\tau)>0$.
Therefore $u(x,\tau)\geq \eta x$ for all $x\in [0,1]$ and
some $\eta\in (0,1)$.
Since ${\mathcal P}[\eta x]\leq 0$ and $\eta<1=u(1,t)$, the comparison principle implies that
\be\label{EqlemA} 
u(x,t)\geq \eta x,\qquad 0\leq x\leq 1,\ t\geq\tau.
\ee
On the other hand, by (\ref{CvToOne}), there exists $T>\tau$ such that 
\be\label{EqlemB}
u(x,t)\geq 1-\delta,\qquad (1-\delta)\delta\leq x\leq 1,\ t\geq T.
\ee
Define 
\be\label{EqlemC}
\underline v(x,t)=\bigl(\eta+2\eta^2(t-T)\bigr)x,\qquad 0\leq x\leq 1,\ t\geq T.
\ee
We have
\be\label{EqlemD}
{\mathcal P}\underline v=2\eta^2x-2\bigl(\eta+2\eta^2(t-T)\bigr)^2x\leq 0
\ee
and, due to (\ref{EqlemA}),
\be\label{EqlemE}
u(x,T)\geq \underline v(x,T),\qquad 0\leq x\leq 1.
\ee
Since $\eta<1<1/\delta$, we may find $T_\delta>T$ such that 
\be\label{EqlemF}
\eta+2\eta^2(T_\delta-T)=1/\delta.
\ee
In view of (\ref{EqlemB}), (\ref{EqlemC}) and (\ref{EqlemF}), we have
\be\label{EqlemG}
u\bigl((1-\delta)\delta,t\bigr)\geq 1-\delta\geq \underline v\bigl((1-\delta)\delta,t\bigr) ,\qquad T\leq t\leq T_\delta.
\ee
It then follows from (\ref{EqlemD}), (\ref{EqlemE}), (\ref{EqlemG}) and the comparison principle that
$$u(x,T_\delta)\geq x/\delta,\qquad 0\leq x\leq (1-\delta)\delta.$$
Using ${\mathcal P}[x/\delta]\leq 0$ and (\ref{EqlemB}), we deduce that
$$u(x,t)\geq x/\delta,\qquad 0\leq x\leq (1-\delta)\delta,\ t\geq T_\delta.$$
This combined with (\ref{EqlemB}) yields the desired conclusion.  \fin

\subsection{ODE lemmas.}
We consider the differential operator
$${\mathcal L}w:=yw''+{2yw'\over 1+y}+{2w\over (1+y)^2}.$$
First, the expression for the operator ${\mathcal L}$ reads
$${\mathcal L}w=\Bigl[\Bigl({y\over y+1}\Bigr)^2\Bigl({(y+1)^2\over y} w\Bigr)'\Bigr]'.$$
If $\psi\in C([0,\infty))$ satisfies $\psi(y)=O(y)$ as $y\to 0$, then the problem
$$
\alignedat2
{\mathcal L}w&=\psi,\quad y>0,\\
w(0)&=0,\ w'(0)=0
\endalignedat
$$
admits a unique solution $w$, and $w$ can be represented as
$$w(y)={\mathcal L}_0^{-1}\psi:={y\over (y+1)^2}\int_0^y\left({t+1\over t}\right)^2\int_0^t\psi(s)ds\,dt$$
(note that the integral is convergent due to the assumption on $\psi$).
In particular,

\be\label{PreservePos}
\psi\geq 0 \text{ on }[0,\infty)\
 \Longrightarrow 
 {\mathcal L}_0^{-1}\psi\geq 0 \text{ on }[0,\infty).
\ee
On the other hand, $w_0(y)={y\over (y+1)^2}$ solves
$$
\alignedat2
{\mathcal L}w&=0,\quad y>0,\\
w(0)&=0,\ w'(0)=1.
\endalignedat
$$

In the next section, to construct our main upper and lower solutions, we will need to know the asymptotic behaviour of the action of the operator $\mathcal{L}_0$ on some particular functions as $y\to \infty.$ 
More precisely,
let us define
$$f= \bigr(I+{\mathcal L}_0^{-1})\Bigl({y\over (y+1)^2}\Bigr),$$
$$\tilde f=2ff'-yf'+f,$$
$$g={\mathcal L}_0^{-1}\tilde f$$
and 
$$h={\mathcal L}_0^{-1}(\tilde f+M\varphi)
=g+M{\mathcal L}_0^{-1}\varphi,$$
where $M>0$ and
$$\varphi\in C^1([0,\infty)),\qquad \varphi(0)=0,\qquad \varphi(y)=
\frac{1}{\log y},\ y\geq 2.$$
We have the following lemmas.

\begin{lem} \label{lem1}
As $y\to\infty$, the function $f$ satisfies
$$f(y)=\log y-2+O\Bigl(\frac{\log^2 y}{y}\Bigr);
\leqno{\rm (i)}$$
$$f'(y)=\frac{1}{y}+O\Bigl(\frac{\log^2 y}{y^2}\Bigr).
\leqno{\rm (ii)}$$
\end{lem}

\begin{lem}\label{lem1b}
As $y\to\infty$, the functions $g$ and $h$ satisfy
$$g(y)=\frac{y\log y}{2}-\frac{9y}{4}+O\bigl(\log^3 y\bigr);
\leqno{\rm (i)}$$
$$g'(y)=\frac{\log y}{2}-\frac{7}{4}+O\Bigl(\frac{\log^3 y}{y}\Bigr);
\leqno{\rm (ii)}$$
$$h(y)=\frac{y\log y}{2}-\frac{9y}{4}+O\Bigl(\frac{y}{\log y}\Bigr);
\leqno{\rm (iii)}$$
$$h'(y)=\frac{\log y}{2}-\frac{7}{4}+O\Bigl(\frac{1}{\log y}\Bigr).
\leqno{\rm (iv)}$$
\end{lem}
{\it Proof of Lemma~\ref{lem1}}\ {(i)\,\,} Using 
\be\label{DLlog}
\log (y+1)=\log y +O(1/y),\quad\mbox{as }y\to \infty,
\ee
we obtain
\begin{eqnarray*}
f(y)&=&\frac{y} {(y+1)^2}\left[1+\int_0^y\left({t+1\over t}\right)^2\int_0^t\left(\frac{1}{s+1}-\frac{1}{(s+1)^2}\right)ds\,dt\right]\\
&=&\frac{y}{(1+y)^2}\left[1+\int_0^y\left\{\left(1+\frac{2}{t}+\frac{1}{t^2}\right)
\left(\log(t+1)-\frac{t}{t+1}\right)\right\}dt \right]\\
&=&\frac{y}{(1+y)^2}\left[y\log y-2y+O\bigl(\log^2 y\bigr)\right]\\
&=&\log y-2+O\left(\frac{\log^2 y}{y}\right),\quad\mbox{as }y\to \infty.
\end{eqnarray*}
{(ii)\,\,} 
Now using $(i)$ we deduce
\begin{eqnarray*}
f'(y)&=&
\left[\frac{1}{y}-\frac{2}{y+1}\right]f(y)+\frac{1}{y}
\left(\log(y+1)-\frac{y}{y+1}\right)\\
&=&\left(-\frac{1} {y}+O\left(\frac{1}{y^2}\right)\right)\left(\log y-2+O\left(\frac{\log^2 y}{y}\right)\right)+
\frac{\log(y+1)}{y}-\frac{1}{y+1}\\
&=&\frac{1}{y}+O\left(\frac{\log^2 y}{y^2}\right),\quad\mbox{as }y\to \infty.
\end{eqnarray*}
\fin

To show Lemma~\ref{lem1b} we first note that, due to Lemma~\ref{lem1} and (\ref{DLlog}), we have
$$\tilde f(y)=f_1(y)-3f_2(y)+O(f_3(y)),\quad\mbox{as } y\to\infty,$$
where
$$f_1(y)=\log(y+1),\quad f_3(y)=\frac{\log^2(y+1)}{y+1}$$
and
$$f_2\in C^1([0,\infty)),\qquad f_2(0)=0,\qquad f_2(y)=1,\ y\geq 1.$$
Denote $g_i={\mathcal L}_0^{-1}f_i$, for $i=1,2,3$, and $g_4={\mathcal L}_0^{-1}\varphi$.
Lemma~\ref{lem1b} will then be an immediate consequence of:

\begin{lem}\label{lem1c}
As $y\to\infty$, the functions $g_i$ satisfy
$$g_1(y)=\frac{y\log y}{2}-\frac{3 y}{4}+O\bigl(\log y\bigr)
\quad\hbox{ and }\quad
g_1'(y)=\frac{\log y}{2}-\frac{1}{4}+O\Bigl(\frac{\log y}{y}\Bigr);
\leqno{\rm (i)}$$
$$g_2(y)=\frac{y}{2}+O(1)
\quad\hbox{ and }\quad
g_2'(y)=\frac{1}{2}+O\Bigl(\frac{1}{y}\Bigr);
\leqno{\rm (ii)}$$
$$g_3(y)=O\Bigl(\log^3 y\Bigr)
\quad\hbox{ and }\quad
g_3'(y)=O\Bigl(\frac{\log^3 y}{y}\Bigr);
\leqno{\rm (iii)}$$
$$g_4(y)=O\Bigl(\frac{y}{\log y}\Bigr)
\quad\hbox{ and }\quad
g_4'(y)=O\Bigl(\frac{1}{\log y}\Bigr).
\leqno{\rm (iv)}$$
\end{lem}

\begin{proof} {(i)\,\,}  As $y\to \infty$, we have
\begin{eqnarray*}
g_1(y)&=&\frac{y} {(y+1)^2}\int_0^y\left({t+1\over t}\right)^2\bigl((t+1)\log (t+1)-t\bigr)dt\\
&=&\frac{y}{(1+y)^2}\int_0^y\bigl[(t+1)\log (t+1)-t+O\left(\log (t+1)\right)\bigr]dt\\
&=&\left[\frac{1}{y}+O\left(\frac{1}{y^2}\right)\right]\left[\frac{y^2\log y}{2}-\frac{3 y^2}{4}+O\left(y\log y\right)\right]\\
&=&\frac{y\log y}{2}-\frac{3 y}{4}+O\left(\log y\right),
\end{eqnarray*}
hence
\begin{eqnarray*}
g_1'(y)&=&
\left[\frac{1}{y}-\frac{2}{y+1}\right]g_1(y)
+\frac{1}{y}\int_0^y \log (s+1) ds\\
&=&\left(-\frac{1} {y}+O\left(\frac{1}{y^2}\right)\right)\left(\frac{y\log y}{2}-\frac{3 y}{4}
+O\left(\log y\right)\right)+\frac{(y+1)\log(y+1)-y}{y}\\
&=&\frac{\log y}{2}-\frac{1}{4}+O\left(\frac{\log y}{y}\right).
\end{eqnarray*}
\item {(ii)\,\,} 
Due to the definition of $g_2$ %REV1: (y) removed
there exist constants $C_1,C_2\in\R$ such that, for $y>1$,
\begin{eqnarray*}
g_2(y)
&=&\frac{y}{(y+1)^2}
\left[ \int_0^1\left({t+1\over t}\right)^2 \int_{0}^{t} f_2(s) ds\,dt
+\int_1^y\left({t+1\over t}\right)^2(C_1+t) dt\right]\\
&=&\frac{y}{(y+1)^2}
\left[C_2+\frac{y^2}{2}+O(y)\right]=\frac{y}{2}+O(1),\quad\mbox{as }y\to \infty.
\end{eqnarray*}
Therefore, for $y>1$,
\begin{eqnarray*}
g_2'(y)
&=&\left[\frac{1}{y}-\frac{2}{y+1}\right]g_2(y)
+\frac{y}{(y+1)^2}\left({y+1\over y}\right)^2 (C_1+y)\\
&=&
\left[\frac{1}{y}-\frac{2}{y+1}\right]g_2(y)
+1+O\left(\frac{1}{y}\right)
=\frac{1}{2}+O\left(\frac{1}{y}\right),\quad\mbox{as }y\to \infty.
\end{eqnarray*}
\item{(iii) \,\,} 
As $y\to \infty$, we have
\begin{eqnarray*}
g_3(y)&=&\frac{y} {3(y+1)^2}\int_0^y\left({t+1\over t}\right)^2\log^3 (t+1)dt\\
&=&\frac{y}{3(1+y)^2}\int_0^y O\left(\log^3 (t+1)\right)dt\\
&=&\frac{y}{3(1+y)^2} \,O\left((y+1)\log^3 (y+1)\right)=O\left(\log^3 y\right),
\end{eqnarray*}
hence
\begin{eqnarray*}
g_3'(y)&=&
\left[\frac{1}{y}-\frac{2}{y+1}\right]g_3(y)
+\frac{1}{y}\int_0^y\frac{\log^2(t+1)}{(t+1)}dt\\
&=&\left(-\frac{1} {y}+O\left(\frac{1}{y^2}\right)\right)O\left(\log^3 y\right)+\frac{\log^3 (y+1)}{3y}
=O\left(\frac{\log^3 y}{y}\right).
\end{eqnarray*}
\item{(iv) \,\,}
Similarly to (ii) in this case we have, for $y>2$,
\begin{eqnarray*}
g_4(y)
&=&\frac{y}{(y+1)^2}
\left[ \int_0^2\left({t+1\over t}\right)^2 \int_{0}^{t}\varphi(s) ds\,dt
+\int_2^y\left({t+1\over t}\right)^2\left(C_1+\int_{2}^{t}\frac{ds}{\log s}\right) dt\right]\\
&=&\frac{y}{(y+1)^2}
\left[C_2+ \int_2^y\left({t+1\over t}\right)^2 O\left(\frac{t}{\log t}\right) dt\right] \\
&=&\frac{y}{(y+1)^2} O\left(\frac{y^2}{\log y}\right)
=O\left(\frac{y}{\log y}\right),\quad\mbox{as }y\to \infty.
\end{eqnarray*}
Therefore, for $y>2$,
$$g_4'(y)
=\left[\frac{1}{y}-\frac{2}{y+1}\right]g_4(y)
+\frac{y}{(y+1)^2}\left({y+1\over y}\right)^2
\left(C_1+\int_{2}^{y}\frac{ds}{\log s}\right)
=O\left(\frac{1}{\log y}\right),\quad\mbox{as }y\to \infty.
$$
\end{proof}
%%%%    %%%%    %%%%    %%%%    %%%%    %%%%    %%%%    %%%%    %%%%    

\section{Proof of the main results}

\subsection{Construction of a lower solution.}
Motivated by the idea of an asymptotic expansion around (moving) steady states and based on a self-similar variable, 
see Subsection~1.3, we make the
following Ansatz:
$$\underline u(x,t)=1-{1\over y+1}+b(t)f(y)-b^2(t)g(y), \qquad y=a(t)x,$$
where the functions $a, b, f, g$ have to be determined.
Note that the variable $y$ now ranges into the time-dependent interval
$[0,a(t)]$.
Here, $a$ and $b$ are expected to satisfy 
$$a(t)\sim u_x(0,t),
\quad \lim_{t\to\infty}a(t)=\infty
\quad\hbox{ and }\quad\lim_{t\to\infty}b(t)=0.$$

\begin{lem}\label{LemMain1}
The problem
\bgee
\underline u_t-x\underline u_{xx}&\leq& 2\underline u{\hskip 1pt}\underline u_x,\quad 0<x<1,\ t>0,\\
\underline u(0,t)&=&0,\quad t\geq 0,\\
\underline u(1,t)&<&1,\quad t\geq 0,\\
\egee
admits a solution of the form
\be\label{AnsatzUnderu}
\underline u(x,t)=1-{1\over y+1}+b(t)f(y)-b^2(t)g(y), \qquad y=a(t)x,
\ee
where the smooth functions $a>0$, $b>0$, $f\geq 0$ and $g$ have the following properties:
$$f(y)\sim \log y,\quad g(y)\sim{y\log y\over 2},\quad\mbox{as }y\to\infty,$$
\be\label{fgLemMain1}
f(0)=g(0)=f'(0)=g'(0)=0,
\ee
\be\label{AsymptbLemA}
a(t)=\Bigl(1+O\bigl(t^{-1/2}\log t\bigr)\Bigr)
\exp\left[\frac{5}{2}+\sqrt{2t}\right],\quad\mbox{as }t\to\infty,
\ee 
\be\label{AsymptbLemB}
b(t)={1+O\bigl(t^{-1/2}\bigr)\over a(t)\log a(t)},\quad\mbox{as }t\to\infty,
\ee 
and moreover
\be\label{underuxpos}
\underline u_x>0,\quad 0\leq x\leq 1,\ t\geq 0.
\ee
\end{lem}

{\it Proof.} {\it Step 1:} {\bf Construction of the subsolution.}
In what follows we shall omit the variables $t$ and/or $y$ when no confusion is likely.
We take $\underline u$ as in (\ref{AnsatzUnderu}) where we assume
\be\label{assum1a}
a>0\quad\hbox{ and }\quad \lim_{t\to\infty}a(t)=\infty.
\ee
We compute
\be\label{underux}
\underline u_x={a\over (1+y)^2}+baf'(y)-b^2ag'(y),
\ee
$$ %%REV2: instead of \label{underuxx}  (not used)
\underline u_{xx}={-2a^2\over (1+y)^3}+ba^2f''(y)-b^2a^2g''(y),
$$
and
$$\underline u_t={a'x\over (1+y)^2}+\Bigl[b'f(y)+ba'xf'(y)-2bb'g(y)-b^2a'xg'(y)\Bigr],$$
hence
$$\underline u_t={a'\over a}{y\over (1+y)^2}
+\Bigl[b'f(y)+{ba'\over a}yf'(y)-2bb'g(y)-{b^2a'\over a}yg'(y)\Bigr].
$$
Recall that the operator ${\mathcal P}$ is defined in (\ref{DefOperP}). It follows that
$$
\aligned
{\mathcal P}\underline u
&={a'\over a}{y\over (1+y)^2}+\Bigl[b'f+{ba'\over a}yf'-2bb'g-{b^2a'\over a}\,yg'\Bigr]\\
&\qquad +2a\,{y\over (1+y)^3}-ba\,yf''+b^2a\,yg''
-2a \Bigl[{y\over 1+y}+bf-b^2g\Bigr] \Bigl[{1\over (1+y)^2}+bf'-b^2g' \Bigr].\\
\endaligned
$$
Collecting terms of same order in $b$ yields
\be\label{Punderu}
\alignedat2
{\mathcal P}\underline u
&=&{a'\over a}{y\over (1+y)^2}+\Bigl[b'f+{ba'\over a}yf'-2bb'g-{b^2a'\over a}\,yg'\Bigr]
-ab \Bigl[yf''+{2f\over (1+y)^2}+{2yf'\over 1+y}\Bigr]\\
&&\qquad+ab^2 \Bigl[yg''+{2g\over (1+y)^2}+{2yg'\over 1+y}-2ff'\Bigr]
+2ab^3 \Bigl[f'g+fg'\Bigr]-2ab^4gg'.
\endaligned
\ee

The natural scaling of the equation leads to the choice
\be\label{Defb}
b:={a'\over a^2},
\ee
so that in the RHS of (\ref{Punderu}), the first term will be of the same order as 
the terms in the second bracket.
Assuming
\be\label{assum2a}
a'>0,
\ee
we also denote 
\be\label{Defgamma}
\gamma:=\Bigl({a\over a'}\Bigr)'
\ee
and observe that
\be\label{Identab}
{a'\over a}=ab,\qquad b'=-(1+\gamma)ab^2,
\ee
where the last equality comes from
$$\gamma=\Bigl({a\over a'}\Bigr)'=\Bigl({1\over ab}\Bigr)'
=-{b'\over ab^2}-{a'\over a^2b}=-{b'\over ab^2}-1.$$
To make things clear, let us already stress that the final choice of $a$ will guarantee
$$\gamma\geq 0\quad\hbox{ and }\quad\lim_{t\to\infty}\gamma(t)=0.$$
Using (\ref{Identab}), identity (\ref{Punderu}) can be recast under the form:
\bgee
{\mathcal P}\underline u
&=&ab\Bigl[{y\over (1+y)^2}-yf''-{2f\over (1+y)^2}-{2yf'\over 1+y}\Bigr]
\\
&&\hskip 3cm
+ab^2 \Bigl[yg''+{2g\over (1+y)^2}+{2yg'\over 1+y}-2ff'
+yf'-(1+\gamma)f\Bigr]\\
&&\hskip 3cm
+ab^3 \Bigl[2f'g+2fg'-yg'+2(1+\gamma)g\Bigr]-2ab^4gg',
\egee
hence
\be\label{PunderuNew}
\alignedat2
{\mathcal P}\underline u
&=& ab\Bigl[{y\over (1+y)^2}-{\mathcal L}f\Bigr]
+ab^2 \Bigl[{\mathcal L}g-2ff'+yf'-(1+\gamma)f\Bigr]\\
&&\qquad\qquad+ab^3 \Bigl[2f'g+2fg'-yg'+2(1+\gamma)g\Bigr]-2ab^4gg'.
\endaligned
\ee
Let us now choose
\be\label{choicef}
f:=(I+{\mathcal L}_0^{-1})\left({y\over (1+y)^2}\right)\geq {y\over (1+y)^2}\geq 0,
\ee
which solves ${\mathcal L}f=y/(1+y)^2$ for $y>0$, with $f(0)=0$ and $f'(0)=1$ (cf.~Subsection~3.4). 
Dividing by $ab^2$, we see that ${\mathcal P}\underline u$ has the same sign as the quantity
$$A
:=\Bigl[{\mathcal L}g-2ff'+yf'-(1+\gamma)f\Bigr]
+b \Bigl[2f'g+2fg'-yg'+2(1+\gamma)g\Bigr]-2b^2gg'.
$$
Next we choose
$$g:={\mathcal L}_0^{-1}\bigl(2ff'-yf'+f\bigr).$$
Therefore
$$A=-\gamma f+b\Bigl[2f'g+2fg'-yg'+2(1+\gamma)g\Bigr]-2b^2gg'.$$

We now proceed to show that the quantity $A$ is nonpositive for large $t$,
by considering separately the
regions $y_0\leq y\leq a(t)$ and $0\leq y\leq y_0$, for some large $y_0$ independent of $t$.
At this point, we make the additional assumptions that
\be\label{equivgamma}
\gamma\sim{1\over \log a}, \quad\mbox{as } t\to\infty
\ee
and
\be\label{equivaprime}
a'\sim {a\over \log a},\quad\hbox{ hence}\quad
b\sim{1\over a\log a}, \quad\mbox{as } t\to\infty
\ee
(which will be verified on the final choice of the function $a$ -- actually more precise
expansions of $\gamma$ and $b$ will be needed in the final matching process).
By Lemmas~\ref{lem1} and \ref{lem1b}, we have
$$f\sim\log y,\ \ f'g=o(g),\ \ fg'=o(g)\ \hbox{ and }\ \ yg'\sim g\sim {y\log y\over 2}, \quad\hbox{as } y\to\infty.$$
Consequently, fixing $\delta>0$ and taking $y_0$ and $t_0$ large enough, we have, for $y_0\leq y\leq a(t)$ and $t\geq t_0$,
$$-\gamma f
\leq (-1+\delta){\log y\over \log a},\qquad g,g'\geq 0,$$
and
$$2f'g+2fg'-yg'+2(1+\gamma)g
\leq \delta\, y\log y-{1\over 2}y\log y
+2(1+\gamma)\Bigl({1\over 2}+\delta\Bigr)y\log y
\leq \Bigl({1\over 2}+4\delta\Bigr)y\log y.$$
Taking $\delta=1/12$ 
and also using (\ref{equivaprime}), we thus obtain
$$
A\leq (-1+\delta){\log y\over \log a}+{1\over a \log a} \Bigl({1\over 2}+5\delta\Bigr)y\log y
= {\log y\over \log a}\Bigl(-1+\delta+\Bigl({1\over 2}+5\delta\Bigr){y\over a}\Bigr)\leq 0
$$
for $y_0\leq y\leq a(t)$ and $t\geq t_0$ (possibly larger). 
Next, for $0\leq y\leq y_0$, (\ref{choicef}) implies 
$$f(y)\geq c_1y,\quad 0\leq y\leq y_0,$$
whereas $f(0)=g(0)=0$ yields
$$|2f'g+2fg'-yg'|+4|g|+|gg'|\leq c_2y,\quad 0\leq y\leq y_0,$$
with $c_1,c_2>0$.
Consequently,
$$A\leq -c_1\gamma(t) y+c_2b(t)y=\Bigl[-c_1+c_2{b(t)\over\gamma(t)}\Bigr]\gamma(t)y\leq 0$$
on $[0,y_0]$ for $t$ large enough, due to (\ref{equivgamma}) and (\ref{equivaprime}).

We have thus proved that, under conditions (\ref{assum1a}), (\ref{assum2a}), (\ref{equivgamma}) and (\ref{equivaprime}),
there holds ${\mathcal P}\underline u\leq 0$ in $(0,1)\times (T,\infty)$ for $T$ large enough.
By a time-shift we may obviously take $T=0$.

\medskip
{\it Step 2.} {\bf Determination of  $\pmb{a(t)}$ by matching at the outer boundary.}
Next, the determination of $a(t)$, will be done by ``matching" with
the boundary condition at $x=1$, i.e. by writing
$$\underline u(1,t)<1,$$ 
which is equivalent to
\be\label{gk2a}
b f(a)-b^2 g(a)<\frac{1}{a+1}.
\ee
It is of course sufficient to check (\ref{gk2a}) for large $t$ 
(thanks to the possibility of shifting time). 
Let us first sketch the resolution of (\ref{gk2a}) in a rough way.
Using (\ref{Defb}) and applying Lemmas~\ref{lem1}(i) and \ref{lem1b}(i) at leading order, we are left with
\be\label{gk2}
\frac{a'}{a^2} \left(\log a- \frac{a'}{a^2} \frac{a \log a}{2}\right)\lesssim\frac{1}{a+1},\quad\mbox{as } t\to\infty. 
\ee
We expect the second term in the bracket of the LHS of the preceding relation to be much smaller than
the first one as $a \to\infty.$ If we ignore it we obtain the differential inequality
\be\label{mk1}
a' \lesssim\frac{a}{\log a},\quad\mbox{as } t\to\infty ,
\ee
which implies that
\bgee
a(t)\lesssim e^{\sqrt{2t}},\quad\mbox{as } t\to\infty.
\egee
However, the latter estimation is not accurate enough to show the desired estimate
and we thus add a correction term in (\ref{mk1}). More precisely, we look for $a(t)$ as the solution of 
\be\label{mk2}
a'=\frac{a}{\log a}(1+\eta),\quad t>0,\qquad a(0)=2, 
\ee 
where the correction term $\eta=\eta(a)$ has the form
\be\label{gk1}
\eta=\frac{5}{2 \log a}+\frac{K}{\log^2 a},\quad K>0.
\ee

Now plugging (\ref{mk2}) into (\ref{gk2a}), 
recalling also (\ref{Defb}), and using the exact asymptotic behaviour of $f(a), g(a)$ as $a \to \infty,$ provided by Lemmas~\ref{lem1}(i) and \ref{lem1b}(i),  we are reduced to the condition
\bgee
\frac{1+\eta}{a \log a} \left[\log a-2+O\left(\frac{\log^2 a}{a}\right)-\frac{1+\eta}{a \log a}\left(\frac{a \log a}{2}-\frac{9 a}{4}+O\left(\log^3 a\right)\right)\right]<\frac{1}{a+1},\quad\mbox{as } t\to\infty
\egee
or
\be\label{dt1}
(1+\eta) \left[1-\frac{5+\eta}{2\log a}+\frac{9}{4\log^2 a}+\frac{9\eta}{4\log^2 a}+O\left(\frac{\log a}{a}\right)\right]<\frac{a}{a+1},
\quad\mbox{as } t\to\infty.
\ee
Plugging (\ref{gk1}) into (\ref{dt1}) we first obtain
\bgee
\left(1+\frac{5}{2 \log a}+\frac{K}{\log^2 a}\right) \left[1-\frac{5}{2\log a}+\frac{1}{\log^2 a}+O\left(\frac{1}{\log^3 a}\right)\right]<1-\frac{1}{a+1},\quad\mbox{as } t\to\infty
\egee
and finally
\be\label{mt1}
1+\frac{4K-21}{4\log^2 a}+O\left(\frac{1}{\log^3 a}\right)<1-\frac{1}{a+1},\quad\mbox{as } t\to\infty.
\ee
In order for (\ref{mt1}) to be satisfied, we choose $K<\frac{21}{4}$ and
we then obtain the following ODE for $a$:
\be\label{tt1}
a'={a\over \log a}\Bigl(1+{5\over 2\log a}+{K\over \log^2 a}\Bigr),\quad t>0,\qquad a(0)=2. 
\ee
Here we should mention that the same form for $\eta,$ given by (\ref{gk1}), will be also considered for the upper solution, with a different constant $K$; see next subsection.
Equation (\ref{tt1}) implies that
$$ %%REV2: instead of \label{tt1b}  (not used)
a'={a\over \log a -5/2}\left(1+O\bigl(\log^{-2} a\bigr)\right),\quad\mbox{as } t\to\infty,
$$
and integrating with respect to $t$ we get
$$\log^2a-5\log a=2t+O\left(\int_0^t\frac{ds}{\log^2(a(s))}\right),\quad\mbox{as } t\to\infty.$$
Solving the quadratic polynomial in $\log a$, and noting that
$\log a(s)\geq \sqrt{2s}$ by (\ref{tt1}), we
end up with (\ref{AsymptbLemA}).
As for (\ref{AsymptbLemB}), it follows from
\be\label{dt2}
b=\frac{a'}{a^2}=\frac{1}{a \log a}\left(1+\frac{5}{2 \log a}+\frac{K}{\log^2 a}\right)
=\frac{1+O\left(\log^{-1}a\right)}{a \log a}
=\frac{1+O\left(t^{-1/2}\right)}{a \log a},\quad\mbox{as } t\to\infty,
\ee
where we used (\ref{tt1}) and $\log a\geq \sqrt{2t},\;\mbox{as}\; t\to\infty$.
On the other hand, denoting $G(s)=s+(5/2)s^2+Ks^3$ and using (\ref{tt1}),
we see that $\gamma=(a/a')'$ satisfies
\be\label{dt3}
\gamma
=\Bigl[\frac{1}{G\bigl(1/\log a\bigr)}\Bigr]'
=\frac{a'}{a\log^2 a}\,\frac{G'\bigl(1/\log a\bigr)}{G^2\bigl(1/\log a\bigr)}
=\frac{G'\bigl(1/\log a\bigr)}{\log^2 a\,G\bigl(1/\log a\bigr)}=H\bigl(1/\log a\bigr),
\ee
where $H(s)=s(1+5s+3Ks^2)(1+(5/2)s+Ks^2)^{-1}$.
Finally, the assumed properties (\ref{assum1a}), (\ref{assum2a}), 
(\ref{equivgamma}) and (\ref{equivaprime}) of $a,b,\gamma$ 
are immediate consequences of (\ref{tt1}), (\ref{dt2}) and (\ref{dt3}).

\medskip

\medskip
{\it Step 3.} {\bf Proof of {\rm {\bf(\ref{underuxpos})}}.}
By (\ref{underux}) we have
$$a^{-1}\underline u_x={1\over (1+y)^2}+bf'(y)-b^2g'(y).$$
By Lemma \ref{lem1b} and (\ref{equivaprime}), taking $y_1$ and $t_1$ large enough,
we have, for $t\geq t_1$ and $y_1\leq y\leq a(t)$,
$$a^{-1}\underline u_x\geq {1\over 2a^2}-{\log y\over a^2\log^2 a}
\geq {1\over a^2}\Bigl({1\over 2}-{1\over \log a}\Bigr)>0.$$
Now, for $0\leq y\leq y_1$ and $t\geq t_1$ possibly larger, we get
$$a^{-1}\underline u_x\geq {1\over (1+y_1)^2}-Cb(t)-Cb^2(t)>0.$$
By a time-shift we may obviously take $t_1=0$ and (\ref{underuxpos}) is proved.  \fin

\medskip
{\bf Remark 4.1.} It is still possible to obtain 
a qualitatively correct lower solution with just a two-term expansion, for instance by making
the simple choice $f(y)=\log (1+y)$, $g=0$ in (\ref{AnsatzUnderu}).
However this yields only the lower grow-up rate up to a multiplicative constant, i.e. $u_x(0,t)\geq Ce^{\sqrt{2t}}$,
and does not enable one to deduce an expansion of the form (\ref{expansionmainthm}).

\bigskip

\subsection{Construction of an upper solution} 
The form (\ref{AnsatzUnderu}) does not seem to be sufficient to construct an accurate
upper solution (i.e., leading to a function $a(t)$ fulfilling 
(\ref{AsymptbLemA})).
We need a slight perturbation, corresponding to the modified Ansatz:
\be\label{AnsatzOveruA}
\overline u(x,t)=1-{1\over y+1}+b(t)f(y)-b^2(t)\tilde g(y,t), \qquad y=a(t)x,
\ee
where 
\be\label{AnsatzOveruB}
\tilde g(y,t)=(1+\eps(t))h(y),
\ee
and $\eps(t)$ goes to $0$ as $t\to\infty$.

\begin{lem}\label{LemMain2}
The problem
\bgee
\overline u_t-x\overline u_{xx}
&\geq& 2{\hskip 1pt}\overline u{\hskip 1pt}\overline u_x,\quad 0<x<1,\ t>0,\\
\overline u(0,t)&=&0,\quad t\geq 0,\\
\overline u(1,t)&\geq&1,\quad t\geq 0,\\
\egee
admits a solution of the form
$$\overline u(x,t)=1-{1\over y+1}+b(t)f(y)-(1+\eps(t))b^2(t)h(y), \qquad y=a(t)x,$$
where the smooth functions $a(t),b(t),f(y),h(y)$ have the following properties:
$$f(y)\sim \log y,\quad h(y)\sim{y\log y\over 2},\quad\mbox{as } y\to\infty,$$
$$f(0)=h(0)=f'(0)=h'(0)=0,$$
\be\label{AsymptbLemA2}
a(t)=\Bigl(1+O\bigl(t^{-1/2}\log t\bigr)\Bigr)
\exp\left[\frac{5}{2}+\sqrt{2t}\right],
\quad\mbox{as } t\to\infty,
\ee
\be\label{AsymptbLemB2}
b(t)={1+O\bigl(t^{-1/2}\bigr)\over a(t)\log a(t)},\quad\mbox{as } t\to\infty,
\ee
\be\label{AsymptbLemC2}
\eps(t)\sim (2t)^{-1/2},\quad\mbox{as } t\to\infty.
\ee
\end{lem}

{\it Proof.} {\it Step 1:} {\bf Construction of the supersolution.}
Taking $\overline u$ as defined in (\ref{AnsatzOveruA}), (\ref{AnsatzOveruB}), the expression for ${\mathcal P}\overline u$ is similar to (\ref{PunderuNew}), except that
$g,g',g''$ are now replaced with $\tilde g, \tilde g_y, \tilde g_{yy}$ and an additional
term $-b^2\eps'h$ is added, which is inherited from $\overline u_t$.
 As in the proof of Lemma~\ref{LemMain1}, $b$ and $\gamma$ are defined through
(\ref{Defb}) and (\ref{Defgamma}), and we assume
(\ref{assum1a}), (\ref{assum2a}), (\ref{equivgamma}) and (\ref{equivaprime}).
This leads to 
\bgee
{\mathcal P}\overline u
&=&ab\Bigl[{y\over (1+y)^2}-{\mathcal L}f\Bigr]
+ab^2 \Bigl[{\mathcal L}\tilde g-2ff'+yf'-(1+\gamma)f-{\eps'\over a} h\Bigr]\\
&&\hskip 6cm
+ab^3 \Bigl[2f'\tilde g+2f\tilde g_y-y\tilde g_y+2(1+\gamma)\tilde g\Bigr]-2ab^4\tilde g\tilde g_y.
\egee
Replacing $\tilde g(y,t)$ with $(1+\eps(t))h(y)$, we obtain
\be\label{Poveru}  
\alignedat2
{\mathcal P}\overline u
&=&ab\Bigl[{y\over (1+y)^2}-{\mathcal L}f\Bigr]
+ab^2 \Bigl[(1+\eps){\mathcal L}h-2ff'+yf'-(1+\gamma)f-{\eps'\over a}h\Bigr]\\
&&\hskip 2cm
+(1+\eps)ab^3 \Bigl[2f'h+2fh'-yh'+2(1+\gamma)h\Bigr]-2(1+\eps)^2ab^4hh'.
\endaligned
\ee
Denote
$$B_0:=(1+\eps){\mathcal L}h-2ff'+yf'-(1+\gamma)f-{\eps'\over a}h.$$
As in Lemma~\ref{LemMain1}, we first choose
$$f:=(I+{\mathcal L}_0^{-1})\Bigl({y\over (1+y)^2}\Bigr).$$
Dividing identity (\ref{Poveru}) by $ab^2$, we see that ${\mathcal P}\overline u$ has the same sign as the quantity
$$B:=B_0+(1+\eps)b \Bigl[2f'h+2fh'-yh'+2(1+\gamma)h\Bigr]-2(1+\eps)^2b^2hh'.$$
Next we choose
$$h={\mathcal L}_0^{-1}\bigl(2ff'-yf'+f+M\varphi\bigr),$$
where $\varphi$ satisfies
\be\label{propphi}
\varphi(0)=0,\qquad \varphi'(0)>0,\qquad \varphi(y)>0\ \hbox{ for }0<y<2,\qquad
\varphi(y)=1/\log y\ \hbox{ for }y\geq 2.
\ee
Note that, taking $M>2$ suitably large, we have
$$2ff'-yf'+f+M\varphi\geq 0,\quad y\geq 0,$$
by Lemma~\ref{lem1}, hence
\be\label{hpos}
h(y)\geq 0,\quad y\geq 0,
\ee
due to (\ref{PreservePos}). We compute
\bgee
B_0
&=&(1+\eps)\bigl(2ff'-yf'+f+M\varphi)-2ff'+yf'-(1+\gamma)f-{\eps'\over a}h\\
&=&\eps\bigl(2ff'-yf')+M(1+\eps)\varphi+(\eps-\gamma)f-{\eps'\over a}h.\\
\egee
At this point, we choose
\be\label{choiceeps}
\eps=\gamma,
\ee
where $\gamma$ is defined by (\ref{Defgamma}), 
and we assume again (\ref{equivgamma}) and (\ref{equivaprime}), along with 
\be\label{gammadecr}
\gamma'\leq 0
\ee
(these assumptions will be verified on the final choice of the function $a$).
By (\ref{hpos}) and (\ref{gammadecr}) we have
$$B_0 \geq\gamma\bigl(2ff'-yf')+M(1+\gamma)\varphi$$
and it follows 
that 
\be\label{signB}
(1+\gamma)^{-1}B
\geq \Bigl[{\gamma\over 1+\gamma}\bigl(2ff'-yf')+M\varphi\Bigr] 
+b \Bigl[2f'h+2fh'-yh'+2(1+\gamma)h\Bigr]-2(1+\gamma)b^2hh'.
\ee

To show that the RHS of (\ref{signB}) is nonnegative for large $t$, 
we again consider separately the
regions $y_0\leq y\leq a(t)$ and $0\leq y\leq y_0$, for some large $y_0$ independent of $t$.
By the estimates in Lemma \ref{lem1}, we have
\be\label{Asymptfh}
f\sim\log y,\ \ f'\sim {1\over y}\ \ \hbox{ and }\ \ yh'\sim h\sim {y\log y\over 2}, \quad\hbox{ as } y\to\infty.
\ee
Consequently,
fixing $\delta>0$, using (\ref{equivgamma})
and taking $y_0$ and $t_0$ large enough
we have, for $y_0\leq y\leq a(t)$ and $t\geq t_0$,

$${\gamma\over 1+\gamma}\bigl(2ff'-yf')+M\varphi \geq -{3\over 2}\gamma+{M\over\log y}
\geq -{2\over \log a}+{M\over\log y}\geq {M-2\over\log y},$$
$$2f'h+2fh'-yh'+2(1+\gamma)h\geq -\Bigl({1\over 2}+\delta\Bigr)y\log y+2(1+\gamma)\Bigl({1\over 2}-\delta\Bigr)y\log y
\geq \Bigl({1\over 2}-4\delta\Bigr)y\log y$$
and 
$$2(1+\gamma)hh'\leq y\log^2 y.$$
Assuming $M\geq 3$, taking $\delta=1/8$, and also using (\ref{equivaprime}), we infer that
$$(1+\gamma)^{-1}B
\geq {1\over\log a}-{y\log^2 y\over a^2\log^2 a}
={1\over\log a}\left(1-{y\log^2 y\over a^2\log a}\right)
\geq{1\over\log a}\left(1-{\log a\over a}\right)\geq 0,
$$
for $y_0\leq y\leq a(t)$ and $t\geq t_0$.
Next, for $0\leq y\leq y_0$, (\ref{propphi}) implies 
$$M\varphi(y)\geq c_1y,$$
whereas $f(0)=h(0)=0$ yields
$$2ff'-yf'\geq -c_2y,\qquad 2f'h+2fh'-yh'+2(1+\gamma)h\geq -c_2y,
\qquad 2(1+\gamma)hh'\leq c_2y,$$
with $c_1,c_2>0$. Therefore,
$$(1+\gamma)^{-1}B\geq -\gamma(t) c_2y+c_1y-b(t)c_2y-b^2(t)c_2y=\Bigl[c_1-c_2\bigl(\gamma(t)+b(t)+b^2(t)\bigr)\Bigr]y\geq 0$$
on $[0,y_0]$ for $t$ large enough.

We have thus proved that ${\mathcal P}\overline u\geq 0$ in $(0,1)\times (T,\infty)$ for $T$ large enough. By a time-shift we may obviously take $T=0$.

\medskip

{\it Step 2.} {\bf Determination of $\pmb{a(t)}$ by matching at the outer boundary.}
The determination of $a(t)$, will be done again by ``matching" with
the boundary condition at $x=1.$ Imposing that
$$\underline u(1,t)\geq 1$$ we obtain
\bgee
b f(a)-b^2(1+\varepsilon) h(a)\geq \frac{1}{a+1},
\egee
and taking (\ref{Defb}) into account we end up with
\be\label{tk1}
\frac{a'}{a^2} \left(f(a)- \frac{a'}{a^2}(1+\varepsilon) h(a)\right)\geq\frac{1}{a+1}.
\ee
Again it suffices to check (\ref{tk1}) for large $t$.
Following the same reasoning as in the case of lower solution we again look for $a(t)$ as a solution of
\be\label{tk2}
a'=\frac{a}{\log a}(1+\eta),\quad\mbox{as } t\to\infty
\ee 
where the correction term $\eta=\eta(a)$ is given by (\ref{gk1})
with $K$ a constant to be determined. 

Plugging (\ref{tk2}) into (\ref{tk1}) and using the exact asymptotic behaviour of $f(a), h(a)$ 
as $a \to \infty,$ given by Lemmas \ref{lem1} and \ref{lem1b}, we arrive at 
$$
\frac{1+\eta}{a \log a} \left[\log a-2+O\left(\frac{\log^2 a}{a}\right)
-\frac{(1+\eta)(1+\varepsilon)}{a \log a}\left(\frac{a \log a}{2}-\frac{9 a}{4}
+O\Bigl(\frac{a}{\log a}\Bigr)\right)\right]\geq \frac{1}{a+1}
$$
or
\be\label{mf1}
(1+\eta) \left[1-\frac{2}{\log a}+O\left(\frac{\log a}{a}\right) \right. \\
-\left.\frac{(1+\eta)(1+\varepsilon)}{\log a}\left(\frac{1}{2}-\frac{9}{4\log a}+O\left(\frac{1}{\log^2 a}\right)\right)\right]\geq \frac{a}{a+1},
\ee
as $t\to\infty$.
Denote by $\Gamma$ the quantity in the bracket in the LHS of (\ref{mf1}).
Using (\ref{gk1}) and (\ref{mf1}) we obtain
\bgee
\Gamma&=&1-\frac{2}{\log a}+O\left(\frac{\log a}{a}\right) \\
&&\qquad-\left(1+\frac{5}{2 \log a}+\frac{K}{ \log^2 a} \right)\left(1+\frac{1}{\log a}+\frac{5}{2 \log^2 a}\right)\left(\frac{1}{2 \log a}-\frac{9}{4 \log^2 a}+O\left(\frac{1}{\log^3 a}\right)\right)\\
&=&1-\frac{5}{2 \log a}+\frac{1}{2 \log^2 a}+O\left(\frac{1}{\log^3 a}\right),\quad\mbox{as }
t\to\infty.
\egee
Then (\ref{mf1}) becomes equivalent to
\bgee
\left(1+\frac{5}{2 \log a}+\frac{K}{ \log^2 a} \right)\left(1-\frac{5}{2 \log a}+\frac{1}{2 \log^2 a}+O\left(\frac{1}{\log^3 a}\right)\right)\geq 1-\frac{1}{a+1},\quad\mbox{as }
t\to\infty
\egee
that is, 
\be\label{mf2}
1+\frac{4 K-23}{4\log^2 a}+O\left(\frac{1}{\log^3 a}\right)\geq 1-\frac{1}{a+1},\quad\mbox{as }
t\to\infty.
\ee

For (\ref{mf2}) to be satisfied we choose $K>\frac{23}{4}$, and
we again take $a$ to be the solution of the ODE (\ref{tt1}).
Then, by the end of Step~2 of the proof of Lemma~\ref{LemMain1}, we obtain
 (\ref{AsymptbLemA2}), (\ref{AsymptbLemB2}), 
 as well as the assumed properties (\ref{assum1a}), (\ref{assum2a}), 
(\ref{equivgamma}) and (\ref{equivaprime}) of $a,b,\gamma$.
Finally, we note that (\ref{equivgamma}), (\ref{AsymptbLemA2}) and (\ref{choiceeps})
guarantee (\ref{AsymptbLemC2}), and that (\ref{dt3}) implies
$$
\gamma'
=\frac{-a'}{a\log^2 a}H'\bigl(1/\log a\bigr)=\frac{-1}{\log^2 a}(H'G)\bigl(1/\log a\bigr)
\sim \frac{-1}{\log^3 a},
\quad \mbox{as } t\to\infty,
$$
hence (\ref{gammadecr}) (after a further time-shift).
\fin

\subsection{Proofs of Theorem~\ref{Theorem 1} and Corollary~\ref{Corollary 2}}

Let $u$ be the solution of (\ref{pbmu1})-(\ref{pbmu4}) and let $\underline u, \overline u$ be the lower/upper solutions 
provided by Lemmas~\ref{LemMain1} and \ref{LemMain2}.
The asymptotic expansion (\ref{expansionmainthm})-(\ref{amainthm}) in Theorem~\ref{Theorem 1}
will be an immediate consequence of the following two Lemmas.
The first one guarantees that $u$ lies between suitable time-shifts of $\underline u$ and $\overline u$.
The second one shows that (shifted versions of) $\underline u$ and $\overline u$ satisfy the 
required asymptotic behaviour.

\begin{lem}
\label{LemFinal1}
(i) There exists $T_1>0$ such that
\be\label{Comp2LemFinal1a}
u(\cdot,t)\geq \underline u(\cdot,t-T_1), \qquad t\geq T_1.
\ee
(ii) Let $\tau$ be as in Lemma~\ref{LemInit1}. Then there exists $T_2>0$ such that
$$u(\cdot,t)\leq \overline u(\cdot,t+T_2), \qquad t\geq \tau.$$
\end{lem}

{\it Proof.} (i) Since $\underline u(\cdot,0)\in C^1([0,1])$
with $\underline u(0,0)=0$, $\underline u(1,0)<1$
and $\underline u_x(\cdot,0)>0$, 
it follows from Lemma~\ref{LemInit2} that $u(\cdot,T_1)\geq \underline u(\cdot,0)$ for some $T_1>0$. 
The assertion then follows from the comparison principle.
\smallskip

(ii) Due to (\ref{Asymptfh}), (\ref{choicef}) and $h(0)=0$, we have
$$f(y)\geq c_1\log(y+1)\quad\hbox{ and }\quad
h(y)\leq c_2\, (y+1)\log(y+1),\qquad y\geq 0,$$
for some $c_1, c_2>0$.
This along with (\ref{AsymptbLemA2})-(\ref{AsymptbLemC2}) implies that, for all $t\geq t_2$ large enough and all $x\in [0,1]$, 
$$
\aligned
\overline u(x,t)
&={ax\over ax+1}+bf(ax)-(1+\eps(t))b^2h(ax)\\
&\geq {ax\over ax+1}+b\log(ax+1)\bigl(c_1-2c_2 (ax+1)b\bigr)\\
&\geq {ax\over ax+1}+b\log(ax+1)\Bigl(c_1-3c_2 {a+1\over a\log a}\Bigr)
\geq {ax\over ax+1}.
\endaligned
$$
Take $\tau,\eta$ as in Lemma~\ref{LemInit1} and set $x_0=1/4K$. 
For all $x\in [0,x_0]$ and $t\geq t_2$, with $t_2\geq \tau$ possibly larger, we have $a/2K\geq ax+1$, hence
$$\overline u(x,t)\geq {ax\over ax+1}\geq 2Kx\geq u(x,\tau),
\qquad 0\leq x\leq x_0,\ t\geq t_2.$$
On the other hand, (\ref{Asymptfh}) implies that 
$$f'(y)\leq c_3\quad\hbox{ and }\quad |h'(y)|\leq c_3+\log(y+1),\qquad y\geq 0,$$
for some $c_3>0$.
For all $x\in (x_0,1]$ and $t\geq t_2$ (possibly larger), we thus have
$$\overline u_x={a\over (1+ax)^2}+baf'(ax)-(1+\eps(t))b^2ah'(ax)
\leq {a\over (1+ax_0)^2}+c_3ba+2(c_3+\log(a+1))b^2a\leq\eta,
$$
due to (\ref{AsymptbLemA2}) and (\ref{AsymptbLemB2}), hence
$$\overline u(x,t)\geq 1-\eta(1-x)\geq u(x,\tau),
\qquad x_0<x\leq 1,\ t\geq t_2.$$
Therefore $\overline u(\cdot,t_2)\geq u(\cdot,\tau)$ in $[0,1].$
The assertion, with $T_2=t_2-\tau$, thus follows from the comparison principle. \fin

\begin{lem}
\label{LemFinal2}
Let 
$$A(t)=\exp\left[\frac{5}{2}+\sqrt{2t}\right].$$
Then, for any $T\in\R$, each of the functions
$w=\underline u$ and $w=\overline u$ satisfies
$$1-w(x,t+T)={1\over 1+A(t)x}\left[1-x+O\bigl(t^{-1/2}\log t\bigr)\right],$$
as $t\to\infty$, uniformly in $[0,1]$.
\end{lem}
{\it Proof.} We shall give the proof for $w=\underline u$, the other case being completely similar.
Set $\tilde a(t)=a(t+T)$ and $\tilde b(t)=b(t+T)$. 
We first note that, by (\ref{AsymptbLemA}), we have
$$\tilde a(t)=\Bigr(1+O\bigl(t^{-1/2}\log t\bigr)\Bigr)A(t)$$
and
\be\label{Esttildea}
\log{\tilde a(t)}=\log A(t)+O\bigl(t^{-1/2}\log t\bigr)=\Bigl(1+O\bigl(t^{-1}\log t\bigr)\Bigr)\log A(t).
\ee
Also, by (\ref{AsymptbLemB}), we have
\be\label{Esttildeb}
\tilde b(t)={1+O\bigl(t^{-1/2}\bigr)\over \tilde a\log{\tilde a}}.
\ee
Moreover, due to Lemma~\ref{lem1}, there exists $C>0$ such that
\be\label{fclosetolog}
|f(y)-\log(1+y)|\leq C,\quad y\geq 0.
\ee

Now, using (\ref{Esttildeb}), (\ref{fclosetolog}) and (\ref{Esttildea}), we compute
$$
\alignedat2
{1\over 1+\tilde ax}-\tilde bf(\tilde ax)
&={1\over 1+\tilde ax}\left[1-{(1+\tilde ax)f(\tilde ax)\over \tilde a\log{\tilde a}}\Bigl(1+O\bigl(t^{-1/2}\bigr)\Bigr)\right]\\
&={1\over 1+\tilde ax}\left[1-{(1+\tilde ax)\log(1+\tilde ax)\over \tilde a\log{\tilde a}}\Bigl(1+O\bigl(t^{-1/2}\bigr)\Bigr)\right]\\
&={1\over 1+\tilde ax}\left[1-{x\log(1+\tilde ax)\over \log{\tilde a}}+O\bigl(t^{-1/2}\bigr)\right]\\
&={1\over 1+\tilde ax}\left[1-x+R(x,t)+O\bigl(t^{-1/2}\bigr)\right],
\endaligned
$$
where
$$R(x,t):={x\bigl(\log{\tilde a}-\log{(1+\tilde ax)}\bigr)\over \log{\tilde a}}.$$
Here and in what follows, the $O$'s are uniform in $[0,1]$.
To control $R$ we note that, if $\tilde ax\geq 1$, then 
$$\log(\tilde ax)\leq \log(1+\tilde ax)\leq \log(\tilde ax)+\log 2,$$
hence 
$$|R(x,t)|\leq {x\bigl(|\log x|+\log 2\bigr)\over \log \tilde a}\leq {C\over \log \tilde a}$$
whereas, if $\tilde ax<1$, then $|R(x,t)|\leq x\leq 1/\tilde a$.
It follows that $\sup_{x\in [0,1]}|R(x,t)|=O(t^{-1/2})$, as $t\to\infty$.
Since, by (\ref{AsymptbLemA}), we have
$$
1+\tilde ax
=(1+Ax)\left[1+{(\tilde a-A)x\over 1+Ax}\right]
=(1+Ax)\left[1+{O\bigl(t^{-1/2}\log t\bigr)Ax\over 1+Ax}\right]
=(1+Ax)\bigl(1+O\bigl(t^{-1/2}\log t\bigr)\bigr),
$$
we deduce that
$$
{1\over 1+\tilde ax}-\tilde bf(\tilde ax)={1+O\bigl(t^{-1/2}\log t\bigr)\over 1+Ax}
\left(1-x+O\bigl(t^{-1/2}\bigr)\right),
$$
hence
\be\label{Est1LemFin2}
{1\over 1+\tilde ax}-\tilde b f(\tilde ax)={1-x+O\bigl(t^{-1/2}\log t\bigr)\over 1+Ax}.
\ee
On the other hand, due to Lemma~\ref{lem1b} and $g(0)=0$, there exists $C_1>0$ such that
$$|g(y)|\leq C_1(1+y)\log(1+y),\quad y\geq 0.$$
Consequently, using also (\ref{Esttildeb}), we obtain, for $t\to\infty$,
\be\label{Est2LemFin2}
(1+Ax){\tilde b}^2|g(\tilde a x)|
\leq 2C_1\,{(1+\tilde ax)^2\log(1+\tilde ax)\over {\tilde a}^2\log^2{\tilde a}}
\leq {3C_1\over \log{\tilde a}}=O\bigl(t^{-1/2}\bigr).
\ee
Combining (\ref{Est1LemFin2}) and (\ref{Est2LemFin2}) finally yields 
$$1-\underline u(x,t+T)={1\over 1+\tilde ax}-\tilde bf(\tilde ax)+{\tilde b}^2g(\tilde a x)
={1-x+O\bigl(t^{-1/2}\log t\bigr)\over 1+Ax}.$$
\fin

{\it Proof of (\ref{reguxThm}) and (\ref{CentralDensity}).}
First notice that estimate (\ref{Est1Init1}) in Lemma~\ref{LemInit1} and 
Lemma~\ref{LemFinal1}(ii) guarantee the control of the slope at $x=1$, namely (\ref{reguxA}),
for any finite $T>0$.
The $C^1$ regularity property (\ref{reguxThm}) is then a consequence of Lemma~\ref{LemC1}. To show 
(\ref{CentralDensity}), note that
$\underline u_x(0,t)=a(t)$ due to (\ref{underux}) and (\ref{fgLemMain1}). 
The lower estimate corresponding
to (\ref{CentralDensity}) is then a consequence of (\ref{Comp2LemFinal1a}) and (\ref{AsymptbLemA}).
The proof of the upper part is similar
by using $\overline u$.\fin

{\it Proof of Corollary~\ref{Corollary 2}.} 
Assertion (i) is an immediate consequence of (\ref{expansionmainthm}) and (\ref{amainthm}). To show (ii), 
it suffices to observe that, due to (\ref{amainthm}),
$$\int_0^1{dx\over 1+A(t)x}=\left[{\log\bigl(1+A(t)x\bigr)\over A(t)}\right]_0^1={\log(1+A(t))\over A(t)}
=\bigl(1+O\bigl(t^{-1/2}\bigr)\bigr)\sqrt{2t}\,\exp\left[-\frac{5}{2}-\sqrt{2t}\right],$$
as $t\to\infty$,
and to use 
$$\int_0^1{x\,dx\over 1+A(t)x}\leq{1\over A(t)}.$$
\fin

\medskip

{\bf Acknowledgements.} This research was performed while N.I.K. was a visitor at the Laboratoire Analyse G\'eom\'etrie et Applications in Universit\'e Paris-Nord. He is grateful to this institution for its hospitality and stimulating atmosphere. N.I.K would like to express his sincere thanks to Piotr Biler who introduced him to the chemotaxis problem and to Andrew A. Lacey for stimulating discussions. 

%%%%    %%%%    %%%%    %%%%    %%%%    %%%%    %%%%    %%%%    %%%

%%%%    %%%%    %%%%    %%%%    %%%%    %%%%    %%%%    %%%%    %%%%    

%%%%    %%%%    %%%%    %%%%    %%%%    %%%%    %%%%    %%%%    %%%%    
\end{document}